\documentclass{article}

\usepackage[final, nonatbib]{neurips_2020}

\usepackage[utf8]{inputenc} 
\usepackage[T1]{fontenc}    
\usepackage{hyperref}       

\usepackage{url}            
\usepackage{booktabs}       
\usepackage{amsfonts}       
\usepackage{nicefrac}       
\usepackage{microtype}      

\usepackage{amsthm,amsmath,amssymb}
\usepackage{times}
\usepackage{graphicx}

\newcommand{\argmin}{\mathop{\rm argmin}}

\newcommand{\less}{\leqslant}
\newcommand{\gre}{\geqslant}

\newcommand{\wtilde}{\widetilde}

\newcommand{\defn}{\ensuremath{: \, =}}

\newcommand{\real}{\ensuremath{\mathbb{R}}}
\newcommand{\Exs}{\ensuremath{\mathbb{E}}}
\newcommand{\trace}{\operatorname{trace}}

\newtheorem{theorem}{Theorem}[section]
\newtheorem{lemma}[theorem]{Lemma}
\newtheorem{remark}[theorem]{Remark}

\title{Optimal Iterative Sketching with the Subsampled Randomized Hadamard Transform}

\author{
  Jonathan Lacotte\thanks{Equal contributions.} \\
  Department of Electrical Engineering\\
  Stanford University\\
  \texttt{lacotte@stanford.edu} \\
   \And
   Sifan Liu$^*$ \\
  Department of Statistics \\
  Stanford University \\
   \texttt{sfliu@stanford.edu} \\
   \AND
  Edgar Dobriban\\
  Department of Statistics\\
  University of Pennsylvania\\
  \texttt{dobriban@wharton.upenn.edu}
   \And
  Mert Pilanci\\
  Department of Electrical Engineering\\
  Stanford University\\
  \texttt{pilanci@stanford.edu}
}

\begin{document}


\maketitle

\begin{abstract}
Random projections or sketching are widely used in many algorithmic and learning contexts. Here we study the performance of iterative Hessian sketch for least-squares problems. By leveraging and extending recent results from random matrix theory on the limiting spectrum of matrices randomly projected with the subsampled randomized Hadamard transform, and truncated Haar matrices, we can study and compare the resulting algorithms to a level of precision that has not been possible before. Our technical contributions include a novel formula for the second moment of the inverse of projected matrices. We also find simple closed-form expressions for asymptotically optimal step-sizes and convergence rates. These show that the convergence rate for Haar and randomized Hadamard matrices are identical, and asymptotically improve upon Gaussian random projections. These techniques may be applied to other algorithms that employ randomized dimension reduction.
\end{abstract}

\section{Introduction}
\label{SectionIntroduction}

Random projections are a classical way of performing dimensionality reduction, and are widely used in many algorithmic and learning contexts, e.g.,~\cite{vempala2005random, mahoney2011randomized, woodruff2014sketching, drineas2016randnla} etc. In this work, we study the performance of the iterative Hessian sketch~\cite{pilanci2016iterative}, in the context of overdetermined least-squares problems
\begin{equation}
\label{EqnMain}
    x^* \defn \argmin_{x \in \real^d} \left\{f(x) \defn \frac{1}{2} \|Ax-b\|^2\right\}\,.
\end{equation}
Here $A \in \mathbb{R}^{n \times d}$ is a given data matrix with $n \gre d$ and $b \in \mathbb{R}^n$ is a vector of observations. For simplicity of notations, we assume throughout this work that $\text{rank}(A)=d$. We will leverage and extend recent results on the limiting spectral distributions of two classical subspace embeddings, random uniform projections and the subsampled randomized Hadamard transform (SRHT), to compare corresponding iterative Hessian sketch versions.

The iterative Hessian sketch (IHS) is an effective iterative method for solving least-squares~\cite{pilanci2015randomized, pilanci2016iterative, lacotte2020effective, sridhar2020lower} (and more general convex optimal optimization problems~\cite{pilanci2017newton}), and it aims to address the condition number dependency of standard iterative solvers as follows. Given step sizes $\{\mu_t\}$ and momentum parameters $\{\beta_t\}$, it computes the update
\begin{equation}
\label{EqnIHSUpdate}
    x_{t+1} = x_t - \mu_t H_t^{-1} \nabla f(x_t) + \beta_t (x_t - x_{t-1})\,,
\end{equation}
where the Hessian $H = A^\top A$ of $f$ is approximated by $H_t = A^\top S_t^\top S_t A$, and $S_0, \dots, S_t, \dots$ are i.i.d.~sketching (random) matrices with dimensions $m \times n$ and $m \ll n$. From now on, we refer to the i.i.d.~property of the sketching matrices as \emph{refreshed} matrices. 

There are many possible choices for the sketching matrices $S_t$, and this is critical for the performance of the IHS. A classical sketch is a matrix $S \in \mathbb{R}^{m \times n}$ with independent and identically distributed (i.i.d.) Gaussian entries $\mathcal{N}(0,m^{-1})$, for which the matrix multiplication $S A$ requires in general $\mathcal{O}(mnd)$ basic operations (using classical matrix multiplication). This is larger than the cost $\mathcal{O}(n d^2)$ of solving~\eqref{EqnMain} with direct methods when $m \gre d$. Another well-studied embedding is the (truncated) $m \times n$ \emph{Haar} matrix $S$, whose rows are orthonormal and with range uniformly distributed among the subspaces of $\real^n$ with dimension $m$. However, this requires time $\mathcal{O}(n m^2)$ to be formed, through a Gram-Schmidt procedure, which is also larger than $\mathcal{O}(nd^2)$. 

The SRHT~\cite{ailon2006approximate,sarlos2006improved} is another classical random orthogonal embedding. Due to the recursive structure of the Hadamard transform, the sketch $SA$ can be formed in $\mathcal{O}(nd \log m)$ time, so that the SRHT is often viewed as a standard reference point for comparing sketching algorithms. 
Moreover, for many applications, random projections with i.i.d.~entries perform worse compared to orthogonal projections~\cite{mahoney2011randomized, mahoney2016structural, drineas2016randnla}. More recently, this observation has also found some theoretical support in limited
contexts~\cite{dobriban2019asymptotics,yang2020reduce}. Works by~\cite{choromanski2017unreasonable} also showed the guaranteed improved performance in accuracy and/or speed. Consequently, along with computational considerations, these results favor the SRHT over Gaussian projections. 

Our goal in this work is to design an optimal version of the IHS with SRHT and Haar embeddings. For this purpose, it is necessary to have a tight characterization of the spectral properties of the matrix $U^\top S^\top S U$ where $U$ is an $n \times d$ partial orthogonal matrix (see, e.g.,~\cite{lacotte2019faster}). With Gaussian embeddings, the matrix $U^\top S^\top S U$ has the well-studied Wishart distribution, see e.g.,~\cite{marchenko1967distribution,anderson2010introduction,tao2012topics,bai2009spectral, couillet2011random,yao2015large}. In fact,~\cite{lacotte2019faster} provided an optimal IHS with Gaussian embeddings, and showed that the best achievable error $\|A(x_t-x^*)\|^2$ scales as $(d/m)^t$. However, a similar analysis does not work for SRHT and Haar sketches. To make progress on this problem, we aim to leverage powerful tools from asymptotic random matrix theory, and \emph{we consider the asymptotic regime where we let the relevant dimensions go to infinity}.

Our technical analysis is based on asymptotic random matrix theory, see e.g.,~ \cite{anderson2010introduction,tao2012topics,bai2009spectral, couillet2011random,yao2015large} etc. Classical results such as the Marchenko-Pastur law do not address well the case of the SRHT, and we leverage recent results on \emph{asymptotically liberating sequences} established by~\cite{anderson2014asymptotically} (see also~\cite{tulino2010capacity} for prior work). Further, we are inspired by the work of~\cite{dobriban2019asymptotics}, who, to our knowledge, first leveraged these results to study the SRHT. However, their results are limited to one-step "sketch-and-solve" methods, and do not address the iterative Hessian sketch. Moreover, while we build on their results, we also need to extend them significantly: for instance, we need to derive the second moment formula for $\theta_{2,h}$ in \eqref{EqnTraceExpressions}, which is novel and non-trivial to establish.

Beyond the IHS, there exist other randomized pre-conditioning methods~\cite{avron2010blendenpik, drineas2011faster, meng2014lsrn, rokhlin2008fast} for solving least-squares, which are based on the SRHT (or closely related sketches) which address effectively the condition number dependency of iterative solvers. Besides least-squares, SRHT sketches are widely used for a wide range of applications across numerical linear algebra, statistics and convex optimization, such as low-rank matrix factorization~\cite{halko2011finding, witten2015randomized}, kernel regression~\cite{yang2017randomized}, random subspace optimization~\cite{lacotte2019high}, or sketch and solve linear regression~\cite{dobriban2019asymptotics}, see the reviews above for applications. Hence, a refined analysis of the SRHT, including our specific technical contributions, may also lead to better algorithms in these fields.

Throughout the paper, we will consistently use the following assumptions and notations for the aspect ratios, $\gamma \defn \lim_{n,d \to \infty} \frac{d}{n} \in (0,1)$, $\xi \defn \lim_{n,m \to \infty} \frac{m}{n} \in (\gamma, 1)$ and $\rho_g \defn \frac{\gamma}{\xi} \in (0,1)$, and the subscript $g$ (resp.~$h$) will refer to Gaussian-related (resp.~Haar and Hadamard-related) quantities. 
We use the notations $\|z\| \equiv \|z\|_2$ for the Euclidean norm of a real vector $z$, $\|M\|_2$ for the operator norm of a matrix $M$, and $\|M\|_F$ for its Frobenius norm. For a sequence of iterates $\{x_t\}$, we denote the error vector $\Delta_t \defn U^\top A(x_t - x^*)$, where $U$ is the $n \times d$ matrix of left singular vectors of $A$. In particular, we have that $\|\Delta_t\|^2 = \|A(x_t-x^*)\|^2$.

\subsection{Overview of our results, contributions and questions left open}

All our contributions hold in the asymptotic limit $n,d,m \to \infty$, and under the aforementioned assumption that the aspect ratios $(d/n)$ and $(m/n)$ have finite limits.

We work with the matrix $U^\top S^\top S U$, where $U$ is an $n \times d$ matrix with orthonormal columns and $S$ is an $m \times n$ Haar or SRHT matrix. Our first results concern Haar projections (Section \ref{SectionHaar}). By leveraging results about their limiting spectral distributions, and after some calculations with Stieljes transforms (defined below) we provide the following new trace formula (see Lemma \ref{TheoremHaarlsd}):
\begin{align*}
    & \theta_{2,h} \defn \lim_{n \to \infty} \frac{1}{d} \, \text{tr}\,\Exs\left[ (U^\top S^\top S U)^{-2}\right] \, = \, \frac{(1-\gamma)(\gamma^2 + \xi - 2\gamma \xi)}{(\xi - \gamma)^3}\,.
\end{align*}
As an application, we characterize explicitly the optimal step sizes $\mu_t$ and momentum parameters $\beta_t$ of the IHS with Haar embeddings (Theorem \ref{TheoremIHSHaar}). We emphasize that the optimal parameters have asymptotically closed form for any data matrix $A$, unlike for certain other propular methods such as gradient descent, which can be useful in practice.
With these optimal parameters, we find that at \emph{any} time step $t \gre 1$ (Theorem \ref{TheoremIHSHaar}),
\begin{align}
    \lim_{n \to \infty} \, \frac{\Exs\|\Delta_t\|^2}{\|\Delta_0\|^2} = \rho_h^t\,,
\end{align}
where the convergence rate $\rho_h$ is given by $\rho_h \defn \rho_g \cdot \frac{\xi (1-\xi)}{\gamma^2 + \xi - 2 \xi \gamma}$, and always satisfies $\rho_h < \rho_g$. By comparing with the prior work \cite{lacotte2019faster}, this  implies that Haar embeddings have uniformly better performance than Gaussian ones.  Further, as an immediate consequence of Theorem~2 in~\cite{lacotte2019faster}, we obtain that the optimal momentum parameters $\beta_t$ are equal to $0$, that is, Heavy-ball momentum does not accelerate the algorithm with refreshed Haar embeddings (Theorem \ref{TheoremIHSHaar} and following discussion). Thus, we are able to characterize explicitly the optimal version of the IHS with Haar embeddings.

Our next results concern SRHT sketches (Section \ref{SectionSRHT}). We prove that under the additional mild assumption on the initial error $\Delta_0$ that $\Exs[\Delta_0 \Delta_0^\top] = d^{-1}I_d$, the IHS with SRHT embeddings also has rate of convergence $\rho_h$ (Theorem \ref{TheoremIHSSRHT}). This relies on novel formulas for the first two inverse moments of SRHT sketches (Lemma \ref{TheoremSRHTlsd}). Consequently, \emph{SRHT matrices  uniformly outperform Gaussian embeddings}. Then, we confirm numerically the above theoretical statements (Section \ref{ns}).

We finally analyze the computational complexity of our method, in comparison to some standard randomized pre-conditioned solvers~\cite{rokhlin2008fast} for dense, ill-conditioned least-squares. We show that in our infinite-dimensional regime, we improve by a factor $\log d$ (Section \ref{comp}).

Importantly, we specifically focus on the IHS with refreshed i.i.d.~embeddings. An immediate variant of the IHS uses the same update~\eqref{EqnIHSUpdate}, but with a fixed embedding $S$ drawn only once at the first iteration, which is appealing in practice. In a concurrent paper~\cite{lacotte2020optimal} more recent to the initial version of the present work, it has been shown that, in the same asymptotic regime, the IHS with a fixed SRHT embedding achieves a better convergence rate. Thus, we emphasize that our core contributions are to develop novel techniques and results for analyzing the IHS with the SRHT, as this may be useful for future developments and extensions of this algorithm in different contexts (e.g., constrained least-squares, convex optimization).

Although we characterize the optimal step sizes and momentum parameters for the IHS with Haar embeddings, we only characterize the optimal step size in the absence of momentum for the IHS with the SRHT. It is thus left as an open question to know whether momentum can accelerate further our method.

\section{Technical Background}
\label{SectionMathematicalBackground}

We introduce a few needed definitions, and we refer the reader to~\cite{bai2009spectral, anderson2010introduction, paul2014random, yao2015large} for an extensive introduction to random matrix theory. Let $\{M_n\}_n$ be a sequence of Hermitian random matrices, where each $M_n$ has size $n \times n$. For a fixed $n$, the empirical spectral distribution (e.s.d.) of $M_n$ is the (cumulative) distribution function of its eigenvalues $\lambda_{1}, \hdots, \lambda_{n}$, i.e., $F_{M_n}(x) \defn \frac{1}{n} \sum_{j=1}^n \mathbf{1}\left\{\lambda_{j} \less x\right\}$ for $x \in \real$, which has density $f_{M_n}(x) = \frac{1}{n} \sum_{j=1}^n \delta_{\lambda_j}(x)$ with $\delta_{\lambda}$ the Dirac measure at $\lambda$. Due to the randomness of the eigenvalues, $F_{M_n}$ is random. The relevant aspect of some classes of large $n \times n$ symmetric random matrices $M_n$ is that, almost surely, the e.s.d.~$F_{M_n}$ converges weakly towards a non-random distribution $F$, as $n \to \infty$. This function $F$, if it exists, will be called the \emph{limiting spectral distribution} (l.s.d.) of $M_n$. 

A powerful tool in the analysis of random matrices is the Stieltjes transform. For $\mu$ a probability measure supported on $[0,+\infty)$, its Stieltjes transform is defined over the complex space complementary to the support of $\mu$ as
\begin{align}
\label{EqnStieltjesTransform}
    m_\mu(z) \defn \int \frac{1}{x-z} \, \mathrm{d}\mu(x)\,.
\end{align}
It holds in particular that $m_\mu$ is analytic over $\mathbb{C}\setminus \real_+$, $m_\mu(z) \in \mathbb{C}^+$ for $z \in \mathbb{C}^+$, $m_\mu(z) \in \mathbb{C}^-$ for $z \in \mathbb{C}^-$ and $\mu_\mu(z) > 0$ for $z < 0$, where $\mathbb{R}_+$ is the set of positive reals and $\mathbb{C}^+$ is the set of complex numbers with positive imaginary part. Another useful transform for studying the product of random matrices is the $S$-transform, denoted $S_\mu$. This is defined as the solution of the following equation, which is unique under certain conditions (see \cite{voiculescu1992free}),
\begin{align}
\label{EqnSTransform}
    m_\mu\!\left(\frac{z+1}{zS_\mu(z)}\right) + z S_\mu(z) = 0.
\end{align}
We introduce a few additional concepts from free probability that will be used in the proofs. We refer the reader to~\cite{voiculescu1992free,hiai2006semicircle,nica2006lectures,anderson2010introduction} for an extensive introduction to this field. Consider the algebra $\mathcal{A}_n$ of $n \times n$ random matrices. For $\!X_n\! \in \mathcal{A}_n$, we define the linear functional $\tau_n(X_n) \!\defn\! \frac{1}{n}\Exs\left[\trace X_n\right]$. Then, we say that a family $\{X_{n,1}, \dots, X_{n,I}\}$ of random matrices in $\mathcal{A}_n$ is \emph{asymptotically free} if for every $i \in \{1,\dots,I\}$, $X_{n,i}$ has a limiting spectral distribution, and if $\tau\left(\prod_{j=1}^m P_j\left(X_{n,i_j}-\tau\left(P_j(X_{n,i_j})\right)\right)\right) \rightarrow 0$ almost surely for any positive integer $m$, any polynomials $P_1, \dots, P_m$ and any indices $i_1, \hdots, i_m \in \{1,\dots,I\}$ with $i_1 \!\neq\!i_2, \dots,i_{m-1}\!\neq\!i_m\neq\!i_1$. In particular, this definition implies that for two sequences of asymptotically free random matrices $X_n,Y_n$, we have the \emph{trace decoupling} relation
\begin{align}
\label{EqnTraceDecoupling}
    \frac{1}{n}\Exs\left[\trace X_n Y_n\right]-\frac{1}{n}\Exs\left[\trace X_n\right]\frac{1}{n}\Exs\left[\trace Y_n\right] \rightarrow 0\,.
\end{align}
Essential to our analysis is the following result. If two $n \times n$ random matrices $A_n$ and $B_n$ are asymptotically free and have respective l.s.d.~$\mu_A$ and $\mu_B$ with respective $S$-transforms $S_A$ and $S_B$, then the matrix product $A_n B_n$ has l.s.d.~$\mu_{AB}$ whose $S$-transform is $S_{AB}(z) = S_A(z)S_B(z)$. The distribution $\mu_{AB}$ is called the free multiplicative convolution of $\mu_A$ and $\mu_B$, and we denote $\mu_{AB} = \mu_A \boxtimes \mu_B$.

We will also make use of an alternative form of the Stieltjes transform: the $\eta$-transform is defined for $z \in \mathbb{C} \setminus \real^-$ as
\begin{align} 
\label{EqnEtaTransform}
    \eta_\mu(z) \defn \int\frac{1}{1+zx}\,\mathrm{d}\mu(x) = \frac{1}{z} m_\mu\left(-\frac{1}{z}\right)\,.
\end{align}
There are standard examples of classes of random matrices and their limiting spectral behavior. We recall a classical result~\cite{marchenko1967distribution}. If $S$ is an $m \times d$ matrix with identically and independently distributed entries $\mathcal{N}(0,1/m)$, then, as $m,d \to \infty$ with $m/d \to \rho \in (0,1)$, the Marchenko-Pastur theorem (see \cite{marchenko1967distribution, bai2009spectral}) states that the matrix $S^\top S$ has l.s.d.~$F_\rho$, whose Stieltjes transform is the unique solution of a certain fixed point equation, and whose density is explicitly given by
\begin{align}
\label{EqnMarcenkoPastur}
    \mu_\rho(x) = \frac{\sqrt{(b-x)_+(x-a)_+}}{2\pi\rho x}\,,
\end{align}
where $y_+=\max\{0,y\}$, $a=(1-\sqrt{\rho})^2$ and $b=(1+\sqrt{\rho})^2$. In our analysis of Haar and SRHT matrices, we will encounter similar fixed-point equations satisfied by the Stieltjes (or $\eta$-) transform of their l.s.d.

\section{Sketching with Haar matrices}
\label{SectionHaar}
Sketching matrices with i.i.d.~entries are not ideal for sketching. Intuitively, i.i.d.~projections distort the geometry of Euclidean space due to their non-orthogonality. In this section, we consider the IHS with refreshed Haar matrices $\{S_t\}$. The following result says that orthogonal projection has better performance than Gaussian projection. 

\begin{theorem}[Optimal IHS with Haar sketches] 
\label{TheoremIHSHaar}
With refreshed Haar matrices $\{S_t\}$, step sizes $\mu_t = \theta_{1,h}/\theta_{2,h}$ (where $\theta_{i,h}$ are defined in Lemma~\ref{TheoremHaarlsd}) and momentum parameters $\beta_t=0$, the sequence of error vectors $\{\Delta_t\}$ satisfies
\begin{align}
\label{EqnErrorIHSHaar}
    \rho_h \defn \left(\lim_{n \to \infty} \frac{\Exs \|\Delta_t\|^2}{\|\Delta_0\|^2}\right)^{1/t} = \rho_g \cdot \frac{\xi (1-\xi)}{\gamma^2 + \xi - 2 \xi \gamma}\,.
\end{align}
Further, for any sequence of step sizes $\{\mu_t\}$ and momentum parameters $\{\beta_t\}$, we have that, for the resulting sequence of error vectors $\{\Delta_t\}$,
\begin{align}
\label{EqnOptimalRateHaar}
    \rho_h \,\less\, \liminf_{t \to \infty} \left(\lim_{n \to \infty} \frac{\Exs \|\Delta_t\|^2}{\|\Delta_0\|^2}\right)^{1/t}\,,
\end{align}
that is, $\rho_h$ is the optimal rate one may achieve using Haar embeddings.
\end{theorem} 

The proof of Theorem~\ref{TheoremIHSHaar}, whose details are deferred to Appendix~\ref{ProofTheoremIHSHaar}, is decomposed into two steps. First, we relate the asymptotic convergence rate $\rho_h$ to the first and second moments of the inverse l.s.d.~of the sketched matrix $SU$, and we adapt to the asymptotic setting the proof of Theorem~1 in~\cite{lacotte2019faster}. Then, and this is our key technical contribution, we provide an explicit formula of this second moment, as given in the following technical lemma.

\begin{lemma}[First two inverse moments of Haar sketches] 
\label{TheoremHaarlsd}
Suppose that $S$ is an $m\times n$ Haar matrix, and let $U$ be an $n\times d$ deterministic matrix with orthonormal columns. It holds that
\begin{align}
\label{EqnTraceExpressions}
    \theta_{1,h} \defn & \lim_{n \to \infty} \frac{1}{d} \, \trace\,\Exs\left[ (U^\top S^\top S U)^{-1}\right] \, = \, \frac{1-\gamma}{\xi - \gamma} \nonumber\\
    \theta_{2,h} \defn & \lim_{n \to \infty} \frac{1}{d} \, \trace\,\Exs\left[ (U^\top S^\top S U)^{-2}\right] \, = \, \frac{(1-\gamma)(\gamma^2 + \xi - 2\gamma \xi)}{(\xi - \gamma)^3}\,.
\end{align}
\end{lemma} 
The formula of the second moment, to the best of our knowledge, is derived explicitly for the first time. We provide a proof sketch here. Note that $\theta_{i,h}$ ($i=1,2$) is the average of the eigenvalues of $U^\top S^\top SU$ to the power of $-i$. Denoting $F_h$ the limiting distribution of the eigenvalues of $U^\top S^\top SU$, we have $\theta_{i,h}=\int x^{-i}dF_h(x)$. This matrix has a specific structure whose l.s.d. has been studied in the random matrix literature. Specifically, given some diagonal non-negative matrices $D,T$ and a squared Haar matrix $W$, Theorem 4.11 of \cite{couillet2011random} characterizes the l.s.d. of matrices of the form $D^{\frac{1}{2}}WTW^\top D^{\frac{1}{2}}$ through a system of functions involving its $\eta$-transform and the l.s.d. of $D,T$. Our setting is more intricate, as $S,U$ are both partial orthogonal matrices, and we need to use an orthogonal complement trick. After getting the $\eta$-transform and thus the Stieltjes transform $m(z)=\int\frac{1}{x-z}dF_h(x)$, we can calculate $\theta_{1,h},\theta_{2,h}$ by evaluating the first and second derivatives of $m(z)$ at $0$. Fortunately in our case, the Stieltjes transform has a closed form, though the calculation is cumbersome. We defer the detailed proof to Appendix~\ref{ProofTraceCalculationsHaar}.

One might wonder how the l.s.d.~of Haar matrices and that of Gaussian embeddings -- the Marchenko-Pastur law $\mu_{\rho_g}$ -- differ. Consider the re-scaled matrix $\frac{n}{m} S_{1,n}^\top S_{1,n}$, whose expectation is equal to the identity. Crucially, the l.s.d.~$\mu_{\rho_g}$ does not depend on the sample size $n$ but only on the limit ratio between $d$ and $m$, whereas the distribution $F_h$ involves the ratios $\gamma$ and $\xi$. Numerically, we observe in Figure~\ref{FigHaarEmpiricalDensity} that, for fixed $\gamma\!=\!0.2$, as $\xi$ increases, the empirical Haar density departs from the Marchenko-Pastur density $\mu_{\rho_g}$, and concentrates more and more relatively to $\mu_{\rho_g}$. Importantly, we see that the support of $F_h$ is included within the support of $\mu_{\rho_g}$, and thus, more concentrated around $1$.
\begin{figure}[h!]
	\centering
	\includegraphics[width=13cm]{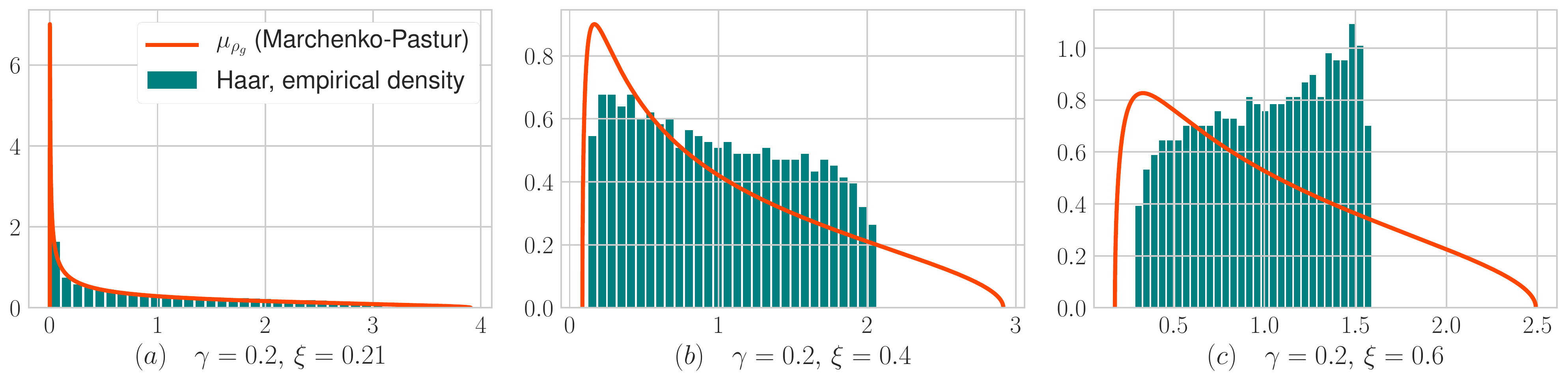}
	\caption{Empirical density of the matrix $\frac{n}{m} U^\top S^\top S U$ for $S$ an $m \!\times\! n$ Haar matrix, versus Marchenko-Pastur density with shape parameter $d/m$. We use $n=4096$, $d=820$ and $m \in \{860, 1640, 2450\}$, so that $\gamma \approx 0.2$ and $\xi \in \{0.2, 0.4, 0.6\}$.}
	\label{FigHaarEmpiricalDensity}
\end{figure}
According to Theorem \ref{TheoremIHSHaar} orthogonal projections are uniformly better than Gaussian i.i.d.~projections. Indeed, the ratio between the convergence rates $\rho_h$ and $\rho_g$ is equal to $\xi (1-\xi) / (\gamma^2 + \xi - 2\gamma \xi)$, and is \emph{always strictly smaller} than $1$. To see this, note that $\xi (1-\xi) / (\gamma^2 + \xi - 2\gamma \xi) < 1$ if and only if $\xi (1-\xi) < \gamma^2 + \xi - 2\gamma \xi$, and after simplification, we obtain the condition $(\xi - \gamma)^2 > 0$. In the small sketch size regime $d \less m \ll n$, we have $\rho_h / \rho_g \approx 1$. As the sketch size $m$ increases relatively to $n$, the convergence rates' ratio scales as $\rho_h / \rho_g \approx (1-\xi)$, and one can improve on the number of iterations -- and thus, data passes -- with Haar embeddings by making $1-\xi$ bounded away from $1$. Further, observe that if we do not reduce the size of the original matrix, so that $m=n$ and $\xi=1$, then the algorithm converges in one iteration. This means that we do not lose any information in the linear model. In contrast, Gaussian projections introduce more distortions than rotation, even though the rows of a Gaussian matrix are almost orthogonal to each other in the high-dimensional setting. The reason is that the eigenvalues are not close to unity.

Interestingly, momentum does not accelerate the refreshed sketch with Haar embeddings. Leveraging past information through the Heavy-ball update~\eqref{EqnIHSUpdate} does not provide any benefit, possibly due to the independence between the sketching matrices $\{S_t\}$. Our proof of this fact is actually an immediate consequence of Theorem~2 in~\cite{lacotte2019faster}. On the other hand, it remains an open question whether there exists a first-order method which uses past iterates along with refreshed matrices, and provide acceleration over gradient descent updates.

We also emphasize that \emph{the optimal parameters have asymptotically closed forms, for any data matrix $A$}! This is quite unexpected and can be useful in practice. The reason is that random projections introduce a great deal of regularity, leading to a "universal" behavior of certain quantities, including those we need. For methods such as gradient descent with momentum, the optimal parameters (e.g, stepsize, momentum), can depend on quantities that can be nontrivial to estimate (e.g, the Lipschitz constant), and require extra computational work.

However, the time complexity of generating an $m \times n$ Haar matrix using the Gram-Schmidt procedure is $O(n m^2)$, which is, for instance, larger than the classical cost $\mathcal{O}(nd^2)$ for solving the least-squares problem~\eqref{EqnMain}, and we now turn to the analysis of another orthogonal matrix, the SRHT, which contains less randomness, but is more structured and faster to generate.

\section{Sketching with SRHT matrices}
\label{SectionSRHT}

We have seen in the previous section that Haar random projections have a better performance than Gaussian i.i.d.~random projections. However, they are still slow to generate and apply. Can we get the same good statistical performance as Haar projections with faster methods? Here we consider the SRHT. This is faster as it relies on the well-structured Walsh-Hadamard transform, which is defined as follows. For an integer $n=2^p$ with $p \gre 1$, the Walsh-Hadamard transform is defined recursively as $H_n = \begin{bmatrix} H_{n/2} & H_{n/2} \\ H_{n/2} & -H_{n/2} \end{bmatrix}$ with $H_1 = 1$. We consider a version of the SRHT which is slightly different from the classical SRHT~\cite{ailon2006approximate}. Our transform $A \mapsto S A$ first \emph{randomly permutes} the rows of $A$, before applying the classical transform. This has negligible cost $\mathcal{O}(n)$ compared to the cost $\mathcal{O}(nd \log m)$ of the matrix multiplication $A \mapsto SA$, and \emph{breaks the non-uniformity} in the data. That is, we define the $n\times n$ subsampled randomized Hadamard matrix as $S = B H_n D P/\sqrt n$, where $B$ is an $n \times n$ diagonal sampling matrix of i.i.d.~Bernoulli random variables with success probability $m/n$, $H_n$ is the $n \times n$ Walsh-Hadamard matrix, $D$ is an $n \times n$ diagonal matrix of i.i.d.~sign random variables, equal to $\pm 1$ with equal probability, and $P\in\real^{n\times n}$ is a uniformly distributed permutation matrix. At the last step, we discard the zero rows of $S$, so that it becomes an $\widetilde m \times n$ orthogonal matrix with $\widetilde m \sim \mathrm{Binomial}(m/n,n)$, and the ratio $\widetilde m / n$ concentrates fast around $\xi$ as $n \to \infty$. Although the dimension $\widetilde m$ is random, we refer to $S$ as an $m \times n$ SRHT matrix.

The following theorem characterizes the exact convergence rate of the IHS with refreshed SRHT embeddings.
\begin{theorem}[IHS with SRHT sketches]
\label{TheoremIHSSRHT}
Suppose that the initial point $x_0$ is random and that the error vector $\Delta_0$ satisfies the condition $\Exs\left[\Delta_0 \Delta_0^\top \right] \!=\! d^{-1} I_d$. Then, with refreshed SRHT matrices $\{S_t\}$, step sizes $\mu_t = \theta_1^h/\theta_2^h$ and momentum parameters $\beta_t = 0$, the sequence of error vectors $\{\Delta_t\}$ satisfies
\begin{align}
    \rho_s \defn \left(\lim_{n \to \infty} \frac{\Exs \|\Delta_t\|^2}{\Exs \|\Delta_0\|^2}\right)^{1/t} = \rho_g \cdot \frac{\xi (1-\xi)}{\gamma^2 + \xi - 2 \xi \gamma} = \rho_h\,. 
\end{align}
\end{theorem} 

Here we impose an additional mild assumption on the initialization of the least-squares problem~\eqref{EqnMain}. We note that the initialization condition $\Exs\left[\Delta_0 \Delta_0^\top \right] \!=\! d^{-1} I_d$ can be achieved by picking $x_0$ uniformly on the unit $d$-sphere $\mathbb{S}^{d-1}$, followed by a uniformly random signed permutation and scaling to the columns of $A$. The key challenge to avoid this is that we need to evaluate
$\Exs \left[\|\Delta_t\|^2\right]= \trace \Exs \left[Q_0\dots Q_{t-1} Q_{t-1} \dots Q_0 \Delta_0 \Delta_0^\top \right]$, where  $Q_t = I_d - \mu_t\,(U^\top S_t^\top S_t U)^{-1}$ and $U$ are the left singular vectors of $A$. Understanding this for general $\Delta_0$ requires properties that are not currently known in random matrix theory (see Appendix~\ref{ProofTheoremIHSSRHT} and Remark~\ref{RemarkOptimalitySRHT} for more details). Further we can only analyze the case $\beta_t=0$, and we do not have a proof for optimality, but \emph{we conjecture that it is true} based on numerical simulations.

We also present an upper-bound on the error, which holds for any deterministic or random initialization $x_0$ and exhibits an identical convergence rate. This is weaker by a factor of $d$, but this is negligible for large $t$.
\begin{theorem}
\label{TheoremIHSSRHTb}
For any initialization $x_0$, with refreshed SRHT matrices $\{S_t\}$, step sizes $\mu_t = \theta_1^h/\theta_2^h$ and momentum parameters $\beta_t = 0$, the sequence of error vectors $\{\Delta_t\}$ satisfies
\begin{align}
   \lim\sup_{n \to \infty} \left(\frac{\Exs \|\Delta_t\|^2}{d\cdot\Exs \|\Delta_0\|^2}\right)^{1/t}  \le  \rho_h\,. 
\end{align}
\end{theorem} 
The proofs of Theorem \ref{TheoremIHSSRHT} and \ref{TheoremIHSSRHTb} are deferred to Appendix~\ref{ProofTheoremIHSSRHT}. While providing significant computational benefits for forming the sketch $SA$, SRHT embeddings are still able to match the convergence rate of orthogonal projections, and thus, also improves on Gaussian sketches. This result follows from the observation that, althouth SRHT has much less randomness than Haar projection, their first two inverse moments behave the same asymptotically. This is formally stated in the following lemma.
\begin{lemma}[First two inverse moments of SRHT sketches] 
\label{TheoremSRHTlsd}
Let $S$ be an $m \times n$ SRHT matrix, $S_h$ be an $m \times n$ Haar matrix, and $U$ an $n \times d$ deterministic matrix with orthonormal columns. Then, the matrices $U^\top S^\top S U$ and $U^\top S_h^\top S_h U$ have the same limiting spectral distribution. Consequently, with $\theta_{1,h},\theta_{2,h}$ as defined in Lemma \ref{TheoremHaarlsd}, it holds that
\begin{align}
\label{EqnTraceExpressionsSRHT}
    &\lim_{n \to \infty} \frac{1}{d} \trace \Exs \left[(U^\top S^\top S U)^{-1} \right] = \theta_{1,h}\,,\\
    &\lim_{n \to \infty} \frac{1}{d} \trace \Exs \left[(U^\top S^\top S U)^{-2} \right] = \theta_{2,h}\,.
\end{align}
\end{lemma} 
The proof is based on recent results about \emph{asymptotically liberating sequences} from the free probability literature \cite{anderson2014asymptotically}, which proves the asymptotic freeness for Hadamard matrices. This technique is also used in \cite{dobriban2019asymptotics} to study SRHT. Specifically, they defined the bi-signed-permutation Hadamard matrix $W=P^\top DHDP$, where $H$ is a Hadamard matrix, $D$ is a sign-flipping diagonal matrix, and $P$ is a permutation. Corollary 3.5, 3.7 of \cite{anderson2014asymptotically} showed that the Bernoulli-sampling diagonal matrix $B$ and $WUU^\top W$ are asymptotically free in the non-commutative probability space of random matrices. Another observation is that, by changing the definition of $S$ to $S=BP^\top DHDP=BW$, the l.s.d. of $U^\top S^\top SU$ remain the same as when $S=BHDP$. The asymptotic freeness shows that the l.s.d. of $U^\top S^\top SU$ for $S$ an SRHT is the same as when $S$ is a Haar matrix. So we get the same results as in Lemma \ref{TheoremHaarlsd}. The detailed proof is defered to Appendix~\ref{sec: moment SRHT}.

In Figure~\ref{FigHaarSRHTEmpiricalDensity}, we verify that the empirical densities with Haar and SRHT matrices are indeed very close.
\begin{figure}[h!]
	\centering
	\includegraphics[width=0.8\linewidth]{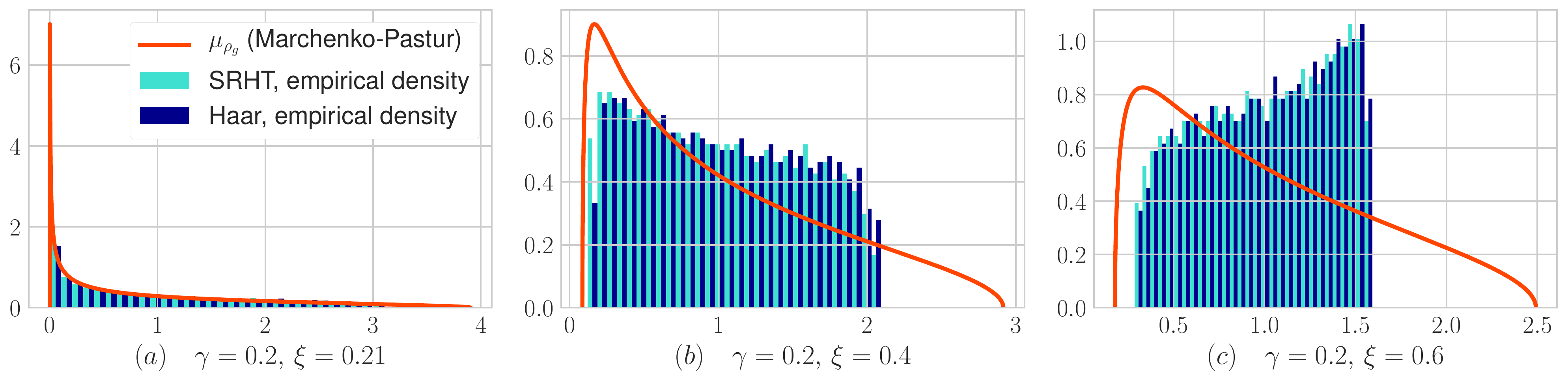}
	\caption{Empirical densities of the matrices $\frac{n}{m} U^\top S^\top S U$ for $S$ an $m \!\times\! n$ Haar matrix and SRHT matrix, versus Marchenko-Pastur density with shape parameter $d/m$. We use $n=4096$, $d=820$ and $m \in \{860, 1640, 2450\}$, so that $\gamma \approx 0.2$ and $\xi \in \{0.21, 0.4, 0.6\}$.}
	\label{FigHaarSRHTEmpiricalDensity}
\end{figure}

\section{Complexity Analysis}
\label{comp}
Let us now turn to a complexity analysis of the IHS with SRHT embeddings, and compare it, in an asymptotic sense, to the complexity of the standard pre-conditioned conjugate gradient method~\cite{rokhlin2008fast}. The latter uses a sketch $SA$ to compute a pre-conditioning matrix $P$, such that $AP^{-1}$ has a small condition number, and then it solves the least-squares problem $\min_y \|AP^{-1}y-b\|^2$, using the conjugate-gradient method. As for the IHS, it can be decomposed into three parts: sketching, factoring (computing $P$ and $AP^{-1}$ versus computing $H_t$), and iterating. The pre-conditioned conjugate gradient prescribes the sketch size $m \approx d \log d$ to guarantee convergence with high-probability. This lower bound is based on the finite-sample bounds on the extremal eigenvalues of the matrix $U^\top S^\top S U$ derived by~\cite{tropp2011improved}. Then, given $\varepsilon \!>\! 0$ and with $m \approx d \log d$, the resulting complexity to achieve $\|\Delta_t\|^2 \less \varepsilon$ scales as $\mathcal{C}_c \asymp nd \log d + d^3 \log d + nd \log(1/\varepsilon)$, where $nd \log d$ is the cost of forming $S A$, the term $d^3 \log d$ is the factoring cost, and $nd \log(1/\varepsilon)$ is the per-iteration cost times the number of iterations. In contrast, we obtain that the IHS with the SRHT can use $m \approx d$, with resulting complexity $\mathcal{C}_n \asymp (nd \log d + d^3 + nd) \log(1/\varepsilon)$. Note that the number of iterations multiplies the sum of the sketching, factoring and per-iteration costs, and this is due to refreshing the sketches. Then, treating the term $\log(1/\varepsilon)$ as a constant independent of the dimensions, we find that, as $n,d,m$ grow to infinity, we have that $C_n/C_c \asymp 1/\log d$.

\section{Numerical Simulations}
\label{ns}

\subsection{Comparison of the different variants of the iterative Hessian sketch}

We evaluate the performance of the IHS with refreshed Haar/SRHT sketches against refreshed Gaussian sketches. 

First, we generate a synthetic data matrix $A \in \real^{n \times d}$ with exponential spectral decay (its $j$-th singular value of $A$ is $\sigma_j = 0.98^j$) and where $n=8192$ and $d=800$. We consider the sketch sizes $m \in \{980, 2450, 4100\}$. For the SRHT, we use the step size $\mu_t=\theta_{1,h}/\theta_{2,h}$ prescribed in Theorem~\ref{TheoremIHSSRHT}, where we replace $\xi$ and $\gamma$ by their finite sample approximations $\xi \approx \frac{m}{n}$ and $\gamma \approx \frac{d}{n}$. For refreshed Gaussian embeddings, we use the optimal parameters $\mu_t$ and $\beta_t$ derived in~\cite{lacotte2019faster}. Results are reported in Figure~\ref{figcomparisonihssynthetic}. As $m$ increases, Haar/SRHT embeddings are increasingly better compared to Gaussian projections. Further, the empirical curves match closely our theoretical predictions: the algorithmic parameters derived from our asymptotic analysis are useful in practice when they are replaced by their finite-sample approximations.
\begin{figure}[h!]
	\centering
	\includegraphics[width=0.9\linewidth]{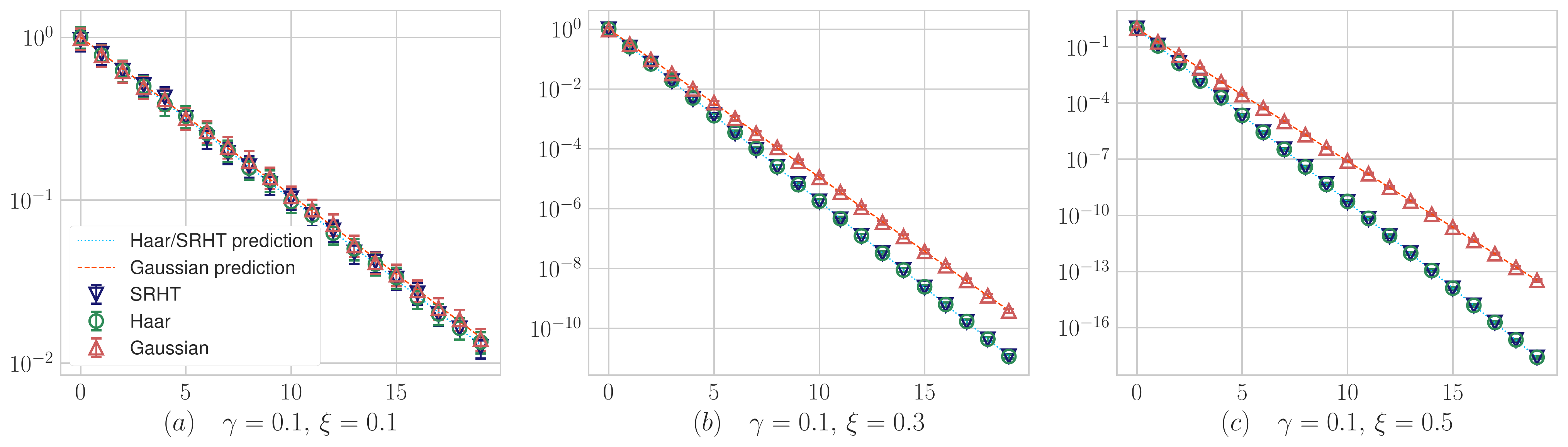}
	\caption{Synthetic dataset: Error $\|\Delta_t\|^2/\|\Delta_0\|^2$ versus number of iterations for the iterative Hessian sketch: (a) $m=980$, (b) $m=2450$ and (c) $m=4100$. We average over $50$ independent trials and empirical standard deviations are shown in the form of error bars.}
	\label{figcomparisonihssynthetic}
\end{figure}
Second, we carry out a similar experiment with the CIFAR10 dataset, for which we consider one-vs-all classification. Here, we have $n=60000$, $d=3072$ and we use the sketch sizes $m \in \{6000, 18000, 30000\}$. Results are reported in~\ref{figcomparisonihscifar}, and we observe similar quantitative results as for the aforementioned synthetic dataset.

\begin{figure}[h!]
	\centering
	\includegraphics[width=0.9\linewidth]{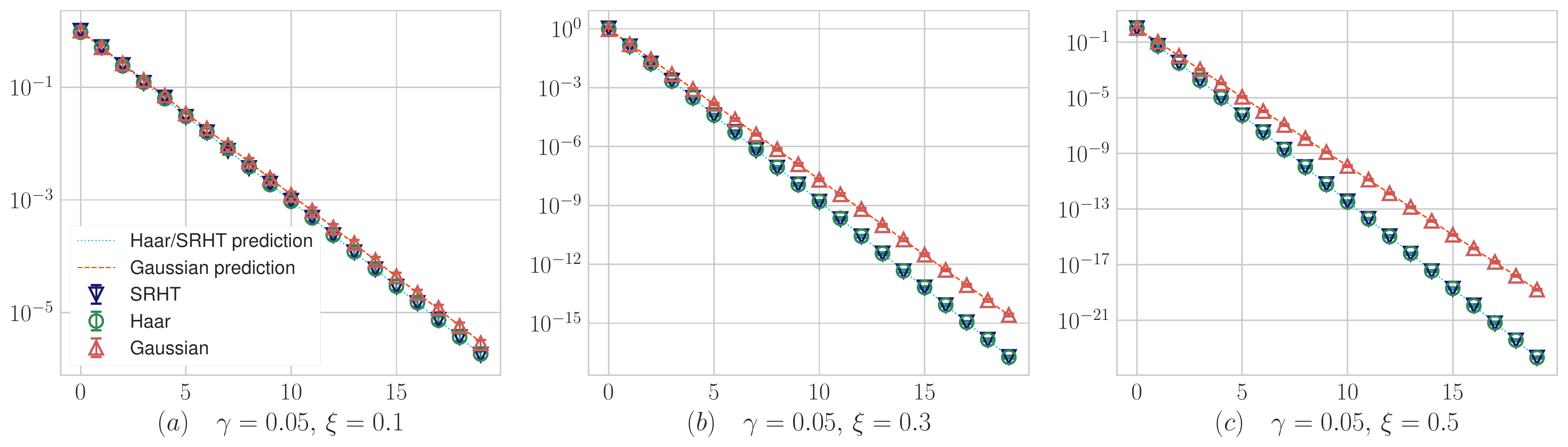}
	\caption{CIFAR10 dataset: Error $\|\Delta_t\|^2/\|\Delta_0\|^2$ versus number of iterations for the iterative Hessian sketch: (a) $m=6000$, (b) $m=18000$ and (c) $m=30000$. We average over $50$ independent trials and empirical standard deviations are shown in the form of error bars.}
    \label{figcomparisonihscifar}
\end{figure}

\subsection{Comparison of the iterative Hessian sketch to standard iterative solvers}

We compare the IHS with the SRHT against the conjugate gradient (CG) method and its preconditioned (pCG) version~\cite{rokhlin2008fast}. We also consider a variant of the IHS, for which we do not refresh the embedding at every iteration. We generate a synthetic data matrix $A \in \real^{n \times d}$ with exponential spectral decay ($\sigma_j = 0.98^j$), $n=4096$ and $d=200$. We consider the sketch sizes $m \in \{1000, 1500, 2000\}$. We observe that the IHS which refreshes embeddings at every iteration has the best convergence rate. More generally, the higher this update frequency, the better the performance. In comparison, CG has the worst convergence rate, which is expected since the data matrix is ill-conditioned, and pCG performs slightly worse than the IHS with update frequency equal to $1$.

\begin{figure}[h!]
	\centering
	\includegraphics[width=0.9\linewidth]{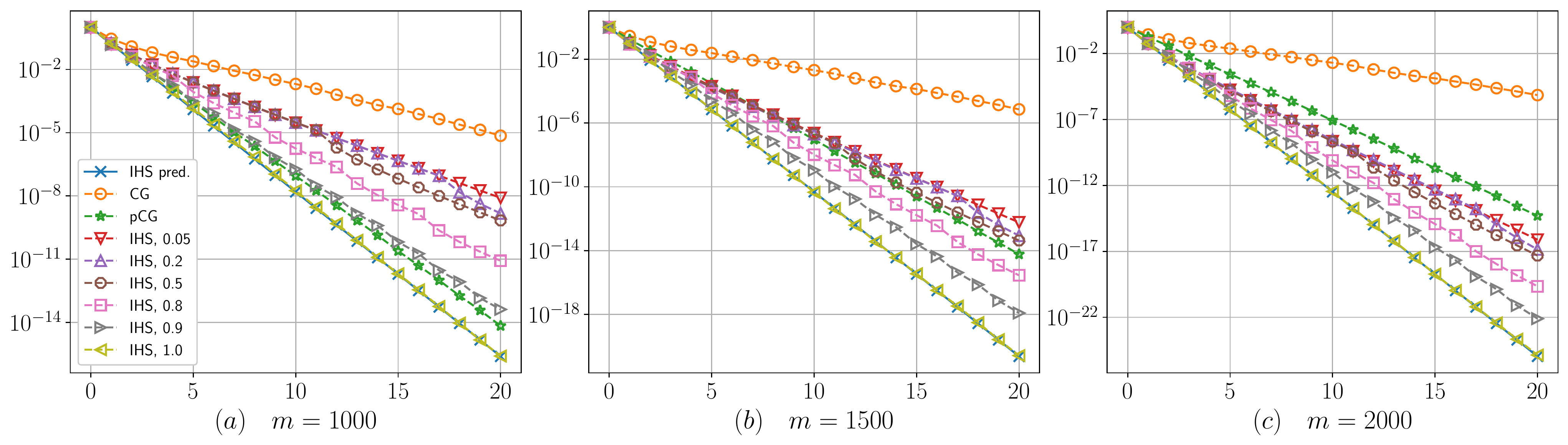}
	\caption{Error $\|\Delta_t\|^2/\|\Delta_0\|^2$ versus number of iterations for the iterative Hessian sketch with the SRHT and different sketch sizes. We average over $50$ independent trials. For instance, 'IHS, $0.2$' refers to the IHS with update frequency equal to $0.2$. For clarity, we do not show error bars for the mean empirical standard deviation which are barely visible.}
	\label{figcomparisonsolvers}
\end{figure}

\begin{ack}
This work was partially supported by the National Science Foundation under grants IIS-1838179 and ECCS-2037304, Facebook Research, Adobe Research and Stanford SystemX Alliance.
\end{ack}


\begin{thebibliography}{10}
	
	\bibitem{ailon2006approximate}
	N.~Ailon and B.~Chazelle.
	\newblock Approximate nearest neighbors and the fast johnson-lindenstrauss
	transform.
	\newblock In {\em Proceedings of the thirty-eighth annual ACM symposium on
		Theory of computing}, pages 557--563. ACM, 2006.
	
	\bibitem{anderson2014asymptotically}
	G.~W. Anderson and B.~Farrell.
	\newblock Asymptotically liberating sequences of random unitary matrices.
	\newblock {\em Advances in Mathematics}, 255:381--413, 2014.
	
	\bibitem{anderson2010introduction}
	G.~W. Anderson, A.~Guionnet, and O.~Zeitouni.
	\newblock {\em An Introduction to Random Matrices}.
	\newblock Number 118. Cambridge University Press, 2010.
	
	\bibitem{avron2010blendenpik}
	H.~Avron, P.~Maymounkov, and S.~Toledo.
	\newblock Blendenpik: Supercharging lapack's least-squares solver.
	\newblock {\em SIAM Journal on Scientific Computing}, 32(3):1217--1236, 2010.
	
	\bibitem{bai2009spectral}
	Z.~Bai and J.~W. Silverstein.
	\newblock {\em Spectral analysis of large dimensional random matrices}.
	\newblock Springer Series in Statistics. Springer, New York, 2nd edition, 2010.
	
	\bibitem{choromanski2017unreasonable}
	K.~M. Choromanski, M.~Rowland, and A.~Weller.
	\newblock The unreasonable effectiveness of structured random orthogonal
	embeddings.
	\newblock In {\em Advances in Neural Information Processing Systems}, pages
	219--228, 2017.
	
	\bibitem{couillet2011random}
	R.~Couillet and M.~Debbah.
	\newblock {\em Random {M}atrix {M}ethods for {W}ireless {C}ommunications}.
	\newblock Cambridge University Press, 2011.
	
	\bibitem{dobriban2019asymptotics}
	E.~Dobriban and S.~Liu.
	\newblock Asymptotics for sketching in least squares regression.
	\newblock In {\em Advances in Neural Information Processing Systems}, pages
	3670--3680, 2019.
	
	\bibitem{drineas2016randnla}
	P.~Drineas and M.~W. Mahoney.
	\newblock Rand{NLA}: randomized numerical linear algebra.
	\newblock {\em Communications of the ACM}, 59(6):80--90, 2016.
	
	\bibitem{drineas2011faster}
	P.~Drineas, M.~W. Mahoney, S.~Muthukrishnan, and T.~Sarl{\'o}s.
	\newblock Faster least squares approximation.
	\newblock {\em Numerische mathematik}, 117(2):219--249, 2011.
	
	\bibitem{halko2011finding}
	N.~Halko, P.-G. Martinsson, and J.~A. Tropp.
	\newblock Finding structure with randomness: Probabilistic algorithms for
	constructing approximate matrix decompositions.
	\newblock {\em SIAM review}, 53(2):217--288, 2011.
	
	\bibitem{hiai2006semicircle}
	F.~Hiai and D.~Petz.
	\newblock {\em The semicircle law, free random variables and entropy}.
	\newblock Number~77. American Mathematical Soc., 2006.
	
	\bibitem{lacotte2019faster}
	J.~Lacotte and M.~Pilanci.
	\newblock Faster least squares optimization.
	\newblock {\em arXiv preprint arXiv:1911.02675}, 2019.
	
	\bibitem{lacotte2020effective}
	J.~Lacotte and M.~Pilanci.
	\newblock Effective dimension adaptive sketching methods for faster regularized
	least-squares optimization.
	\newblock {\em arXiv preprint arXiv:2006.05874}, 2020.
	
	\bibitem{lacotte2020optimal}
	J.~Lacotte and M.~Pilanci.
	\newblock Optimal randomized first-order methods for least-squares problems.
	\newblock {\em arXiv preprint arXiv:2002.09488}, 2020.
	
	\bibitem{lacotte2019high}
	J.~Lacotte, M.~Pilanci, and M.~Pavone.
	\newblock High-dimensional optimization in adaptive random subspaces.
	\newblock In {\em Advances in Neural Information Processing Systems}, pages
	10846--10856, 2019.
	
	\bibitem{mahoney2011randomized}
	M.~W. Mahoney.
	\newblock Randomized algorithms for matrices and data.
	\newblock {\em Foundations and Trends{\textregistered} in Machine Learning},
	3(2):123--224, 2011.
	
	\bibitem{mahoney2016structural}
	M.~W. Mahoney and P.~Drineas.
	\newblock Structural properties underlying high-quality randomized numerical
	linear algebra algorithms., 2016.
	
	\bibitem{marchenko1967distribution}
	V.~A. Marchenko and L.~A. Pastur.
	\newblock Distribution of eigenvalues for some sets of random matrices.
	\newblock {\em Mat. Sb.}, 114(4):507--536, 1967.
	
	\bibitem{meng2014lsrn}
	X.~Meng, M.~A. Saunders, and M.~W. Mahoney.
	\newblock Lsrn: A parallel iterative solver for strongly over-or
	underdetermined systems.
	\newblock {\em SIAM Journal on Scientific Computing}, 36(2):C95--C118, 2014.
	
	\bibitem{nica2006lectures}
	A.~Nica and R.~Speicher.
	\newblock {\em Lectures on the combinatorics of free probability}, volume~13.
	\newblock Cambridge University Press, 2006.
	
	\bibitem{paul2014random}
	D.~Paul and A.~Aue.
	\newblock Random matrix theory in statistics: A review.
	\newblock {\em Journal of Statistical Planning and Inference}, 150:1--29, 2014.
	
	\bibitem{pilanci2015randomized}
	M.~Pilanci and M.~J. Wainwright.
	\newblock Randomized sketches of convex programs with sharp guarantees.
	\newblock {\em IEEE Transactions on Information Theory}, 61(9):5096--5115,
	2015.
	
	\bibitem{pilanci2016iterative}
	M.~Pilanci and M.~J. Wainwright.
	\newblock Iterative hessian sketch: Fast and accurate solution approximation
	for constrained least-squares.
	\newblock {\em The Journal of Machine Learning Research}, 17(1):1842--1879,
	2016.
	
	\bibitem{pilanci2017newton}
	M.~Pilanci and M.~J. Wainwright.
	\newblock Newton sketch: A near linear-time optimization algorithm with
	linear-quadratic convergence.
	\newblock {\em SIAM Journal on Optimization}, 27(1):205--245, 2017.
	
	\bibitem{rokhlin2008fast}
	V.~Rokhlin and M.~Tygert.
	\newblock A fast randomized algorithm for overdetermined linear least-squares
	regression.
	\newblock {\em Proceedings of the National Academy of Sciences},
	105(36):13212--13217, 2008.
	
	\bibitem{sarlos2006improved}
	T.~Sarlos.
	\newblock Improved approximation algorithms for large matrices via random
	projections.
	\newblock In {\em Foundations of Computer Science, 2006. FOCS'06. 47th Annual
		IEEE Symposium on}, pages 143--152. IEEE, 2006.
	
	\bibitem{sridhar2020lower}
	S.~Sridhar, M.~Pilanci, and A.~{\"O}zg{\"u}r.
	\newblock Lower bounds and a near-optimal shrinkage estimator for least squares
	using random projections.
	\newblock {\em arXiv preprint arXiv:2006.08160}, 2020.
	
	\bibitem{tao2012topics}
	T.~Tao.
	\newblock {\em Topics in Random Matrix Theory}, volume 132.
	\newblock American Mathematical Soc., 2012.
	
	\bibitem{tropp2011improved}
	J.~A. Tropp.
	\newblock Improved analysis of the subsampled randomized hadamard transform.
	\newblock {\em Advances in Adaptive Data Analysis}, 3(01n02):115--126, 2011.
	
	\bibitem{tulino2010capacity}
	A.~M. Tulino, G.~Caire, S.~Shamai, and S.~Verd{\'u}.
	\newblock Capacity of channels with frequency-selective and time-selective
	fading.
	\newblock {\em IEEE Transactions on Information Theory}, 56(3):1187--1215,
	2010.
	
	\bibitem{vempala2005random}
	S.~S. Vempala.
	\newblock {\em The random projection method}, volume~65.
	\newblock American Mathematical Soc., 2005.
	
	\bibitem{voiculescu1992free}
	D.~V. Voiculescu, K.~J. Dykema, and A.~Nica.
	\newblock {\em Free random variables}.
	\newblock Number~1. American Mathematical Soc., 1992.
	
	\bibitem{witten2015randomized}
	R.~Witten and E.~Candes.
	\newblock Randomized algorithms for low-rank matrix factorizations: sharp
	performance bounds.
	\newblock {\em Algorithmica}, 72(1):264--281, 2015.
	
	\bibitem{woodruff2014sketching}
	D.~P. Woodruff.
	\newblock Sketching as a tool for numerical linear algebra.
	\newblock {\em Foundations and Trends{\textregistered} in Theoretical Computer
		Science}, 10(1--2):1--157, 2014.
	
	\bibitem{yang2020reduce}
	F.~Yang, S.~Liu, E.~Dobriban, and D.~P. Woodruff.
	\newblock How to reduce dimension with pca and random projections?
	\newblock {\em arXiv preprint arXiv:2005.00511}, 2020.
	
	\bibitem{yang2017randomized}
	Y.~Yang, M.~Pilanci, and M.~J. Wainwright.
	\newblock Randomized sketches for kernels: Fast and optimal nonparametric
	regression.
	\newblock {\em The Annals of Statistics}, 45(3):991--1023, 2017.
	
	\bibitem{yao2015large}
	J.~Yao, Z.~Bai, and S.~Zheng.
	\newblock {\em Large Sample Covariance Matrices and High-Dimensional Data
		Analysis}.
	\newblock Cambridge University Press, New York, 2015.
	
\end{thebibliography}


\begin{appendix}
	\section{Proofs of main theorems}
	\label{appen}
	
	\subsection{Calculations of $\theta_{1,h}$ and $\theta_{2,h}$ for Haar sketch}
	\label{ProofTraceCalculationsHaar}
	We first prove some lemmas and provide the proof of \ref{TheoremHaarlsd} in Section \ref{prf: prf first second moment}.
	
	This lemma characterizes the Stieltjes transform of the l.s.d. of $S_nU_n$.
	\begin{lemma}[Stieltjes transform of l.s.d. of $S_nU_n$]
		\label{TheoremStieltjesTransformHaar}
		We set $S_{1,n}=S_nU_n$. Then the matrix $S_{1,n}^\top S_{1,n}$ admits a l.s.d.~whose Stieltjes transform $m_h$ is given by
		\begin{align}
		\label{EqnStieltjesTransformHaar}
		m_h(z) = \frac{z (2\gamma-1) + \xi - \gamma - \sqrt{(\gamma+\xi-2+z)^2 + 4(z-1)(1-\gamma)(1-\xi)}}{2 \gamma z (1-z)}\,,
		\end{align}
		for any $z \in \mathbb{C} \setminus \real_+$.
	\end{lemma}
	\begin{proof}
		First, observe that since both $S_n$ and $U_n$ are rectangular orthogonal matrices, we can embed them into full orthogonal matrices as $\mathbb{S}_n = \left(\begin{array}{c}S_n \\ S_n^\perp\end{array}\right)$ and $\mathbb{U}_n = \left(\begin{array}{cc}U_n & U_n^\perp\end{array}\right)$. Then, we can write
		\begin{align}
		\label{EqnFactoredOrtho}
		S_{1,n} = \left(\begin{array}{cc}I_m & 0\end{array}
		\right)\mathbb{S}_n\mathbb{U}_n\left(\begin{array}{c}I_d \\ 0 \end{array}
		\right)\,.
		\end{align}
		Let $\mathbb{W}_n=\mathbb{S}_n\mathbb{U}_n$, which is an $n\times n$ Haar matrix due to the orthogonal invariance of the Haar distribution. Then, we define
		\begin{align}
		\label{EqnCn}
		C_n \defn \left(\begin{array}{cc}S_{1,n} S_{1,n}^\top & 0\\0 & 0\end{array}\right) = \left(\begin{array}{cc}I_m & 0\\0 & 0\end{array}\right)
		\mathbb{W}_n
		\left(\begin{array}{cc} I_d & 0\\0&0\end{array}\right)
		\mathbb{W}_n^\top
		\left(\begin{array}{cc}I_m & 0\\0 & 0\end{array}\right)\,.
		\end{align}
		The matrix $C_n$ is related to our matrix of interest $S_{1,n}^\top S_{1,n}$, as they have exactly the same non-zero eigenvalues. Thus, as a first step to establish Lemma~\ref{TheoremStieltjesTransformHaar}, we characterize the l.s.d.~of $C_n$.
		
		%
		The matrix $C_n$ admits a l.s.d.~$F_C$, whose Stieltjes transform $m_C$ is given by
		\begin{align}
		\label{EqnStietljesC}
		m_C(z) = \frac{z+\gamma+\xi-2-\sqrt{(\gamma+\xi-2+z)^2+4(z-1)(1-\gamma)(1-\xi)}}{2z(1-z)}\,,
		\end{align}
		for any $z \in \mathbb{C} \setminus \real_+$.
		%
		%
		The above expression~\eqref{EqnCn} of the matrix $C_n$ has the required form to apply Theorem 4.11 by~\cite{couillet2011random}, and hence characterize the e.s.d.~of $C_n$ through its $\eta$-transform which has to satisfy a fixed-point equation. We defer details of the proof to Section \ref{sec: esd of Cn}.
		Now, we use the fact that the matrices $S_{1,n}^\top S_{1,n}$ and $C_n$ have the same non-zero eigenvalues. Almost surely, there are exactly $d$ of them, which we denote $\lambda_1,\dots,\lambda_d$. Then, the e.s.d.~$F_{C_n}$ of $C_n$ can be decomposed as
		\begin{align}
		\label{EqnESDCn}
		F_{C_n}(x) &= \left(1-\frac{d}{n}\right) \mathbf{1}_{\{x \gre 0\}} + \frac{1}{n} \sum_{i=1}^d \mathbf{1}_{\{x \gre \lambda_i\}} = \left(1-\frac{d}{n}\right) \mathbf{1}_{\{x \gre 0\}} + \frac{d}{n} \, F_{h,n}(x)\,,
		\end{align}
		where $F_{h,n}$ is the e.s.d.~of $S_{1,n}^\top S_{1,n}$. Taking the limit $n \to \infty$, we find that $F_{1,n}$ converges weakly almost surely to
		\begin{align}
		\label{EqnFhFc}
		F_h(x) = \frac{1}{\gamma} \left( F_{C}(x) - (1-\gamma)\mathbf{1}_{\{x \gre 0\}} \right)\,.
		\end{align}
		By definition of $m_h$ and using~\eqref{EqnFhFc}, it follows that for $z \in \mathbb{C} \setminus \real_+$
		\begin{align}
		m_h(z) = \int \frac{1}{x-z} \, \mathrm{d}F_h(x) &= \frac{1}{\gamma} \int \frac{1}{x-z} \, \mathrm{d}F_C(x) - \frac{1-\gamma}{\gamma} \int \frac{1}{x-z} \, \delta_0(x) \mathrm{d}x\\
		&= \frac{1}{\gamma} m_C(z) + \frac{1-\gamma}{\gamma z}\,.
		\end{align}
		Plugging-in the expression of $m_C$, we obtain the claimed formula~\eqref{EqnStieltjesTransformHaar} for $m_h$.
		
	\end{proof}
	
	We will need the following result regarding the support of $F_h$, which is proved in Appendix~\ref{ProofLemmaSupportFh}.
	\begin{lemma}
		\label{lem: support of Fh}
		The support of $F_h$ satisfies
		\begin{align}
		\label{EqnLowerBoundSupportFh}
		\inf \mathrm{supp}(F_h) \gre \frac{(1-\sqrt{\rho_g})^2}{\left(1+\frac{1}{\sqrt{\xi}}\right)^2}\,.
		\end{align}
	\end{lemma}
	Thus, the support of $F_h$ is bounded away from $0$, so is the intersection of the support of $F_C$ and $\real^*$. Further, the distribution $F_C$ has a point mass at $0$ equal to $1-\gamma$. We now turn to the trace calculations in Lemma \ref{TheoremHaarlsd}.

	\subsubsection{Proof of Lemma \ref{TheoremHaarlsd}}
	\label{prf: prf first second moment}
	\begin{enumerate}
		\item \textbf{Computing $\theta_{1,h}$}
		
		Using the facts that $F_C$ has support within $[0,+\infty)$ and a point mass equal to $(1-\gamma)$ at $0$, its $\eta$-transform $\eta_C$ is well-defined on $\{z \in \real \mid z > 0\}$, and, for $z > 0$, it can be decomposed as
		\begin{align}
		\label{Eqn1Theta1}
		\eta_C(z) = 1-\gamma + \int_{x\neq0}\frac{1}{1+zx} \mathrm{d}F_C(x)\,.
		\end{align}
		The function $\frac{1}{x}$ is integrable on the set $\{x>0\}$ with respect to $F_C$, since the support of $F_C$ on $\real^*$ is bounded away from $0$. Since $|\frac{z}{1+xz}|<\frac{1}{x}$ when $z>0, x>0$, it follows by the dominated convergence theorem that
		\begin{align}
		\lim_{z\rightarrow\infty}\int_{x\neq0}\frac{z}{1+xz} \mathrm{d}F_C(x) = \int_{x\neq0} \lim_{z\rightarrow\infty} \frac{z}{1+xz} \mathrm{d}F_C(x) = \int_{x\neq0}\frac{1}{x} \mathrm{d}F_C(x)\,.
		\end{align}
		Using~\eqref{Eqn1Theta1}, it follows that
		\begin{align}
		\label{EqnLimit1}
		\lim_{z\rightarrow\infty}z \left(\eta_C(z)-(1-\gamma)\right) = \int_{x \neq 0} \frac{1}{x} \, \mathrm{d}F_C(x)\,,
		\end{align}
		On the other hand, we have that
		\begin{align}
		\lim_{z\rightarrow\infty}\eta_C(z) &= (1-\gamma) + \lim_{z\rightarrow\infty}\int_{x\neq0}\frac{1}{1+zx}\,\mathrm{d}F_C(t)\\
		&= (1-\gamma) + \int_{x\neq0}\lim_{z\rightarrow\infty}\frac{1}{1+zx}\, \mathrm{d} F_C(x)\\
		&=1-\gamma \label{EqnLimit2}\,.
		\end{align}
		where the second equality is again justified by the dominated convergence theorem. Subtracting $1-\gamma$ from both sides of \eqref{EqnetaCfixedpoint}, multiplying by $z\left(1+\frac{\xi-1}{\eta_C(z)}\right)$ and letting $z \rightarrow \infty$, we obtain
		\begin{align*}
		\lim_{z\rightarrow\infty} z\left(1+\frac{\xi-1}{\eta_C(z)}\right)\left(\eta_C(z)-(1-\gamma)\right)&=\lim_{z\rightarrow\infty}z \left(1+\frac{\xi-1}{\eta_C(z)}\right)
		\left(\frac{\gamma}{1+z(1+\frac{\xi-1}{\eta_C(z)})}\right)\,.
		\end{align*}
		Note that the right-hand side of the above equation is equal to $\gamma$, and the left-hand side satisfies
		\begin{align*}
		\lim_{z\rightarrow\infty} z \left(1+\frac{\xi-1}{\eta_C(z)}\right) \left(\eta_C(z)-(1-\gamma)\right)
		&=\lim_{z\rightarrow\infty} z \left(\eta_C(z)-(1-\gamma)\right) \left(1+\frac{\xi-1}{1-\gamma}\right)\\
		&= \frac{\xi-\gamma}{1-\gamma} \cdot \int_{x\neq0} \frac{1}{x}\,\mathrm{d}F_C(x),
		\end{align*}
		where we used~\eqref{EqnLimit1} and~\eqref{EqnLimit2}. This shows that $\gamma=\frac{\xi-\gamma}{1-\gamma}\int_{x\neq0}\frac{1}{x}\, \mathrm{d}F_C(x)$. We conclude by observing that
		\begin{align*}
		\theta_{1,h} = \lim_{n \to \infty} \frac{1}{d} \trace \Exs\left[(S_{1,n}^\top S_{1,n})^{-1}\right] = \frac{1}{\gamma} \cdot \lim_{n \to \infty} \Exs\left[ \frac{1}{n} \sum_{i=1}^d \frac{1}{\lambda_i} \right] = \frac{1}{\gamma} \int_{x \neq 0} \frac{1}{x} \, \mathrm{d}F_C(x)\,,
		\end{align*}
		and consequently, $\theta_{1,h} = \frac{1-\gamma}{\xi - \gamma}$, which is the claimed result.

		\item \textbf{Computing $\theta_{2,h}$}
		
		Unrolling its definition, we have that
		\begin{align*}
		\theta_{2,h} = \lim_{n \to \infty} \frac{1}{d} \trace \Exs\left[ (S_{1,n}^\top S_{1,n})^{-2} \right] = \frac{1}{\gamma} \cdot \lim_{n \to \infty} \Exs\left[\frac{1}{n} \sum_{i=1}^d \frac{1}{\lambda_i^2}\right] = \frac{1}{\gamma} \int_{\{x\neq0\}}\frac{1}{x^2}\, \mathrm{d}F_C(x)\,,
		\end{align*}
		where the limit in the third equation holds and is finite since $F_C$ has support bounded away from $0$ on $\real^*$. By definition of $m_C$ and using the fact that $F_C$ has point mass $1-\gamma$ at $0$, we get that 
		\begin{align*}
		\frac{\mathrm{d}m_C(z)}{\mathrm{d}z}=\int\frac{1}{(x-z)^2} \, \mathrm{d} F_C(x)=\frac{1-\gamma}{z^2}+\int_{\{x\neq0\}}\frac{1}{(x-z)^2} \, \mathrm{d}F_C(x)\,.
		\end{align*}
		Using again the fact that $F_C$ has support bounded away from $0$ on $\real^*$ and the dominated convergence theorem, we have that $\gamma \theta_{2,h} = \lim_{z \to 0} \int_{x \neq 0}\frac{1}{(x-z)^2} \, \mathrm{d}F_C(x)$, and thus,
		\begin{align*}
		\gamma \theta_{2,h} = \lim_{z\rightarrow0} \left\{\frac{\mathrm{d}m_C(z)}{\mathrm{d}z}-\frac{1-\gamma}{z^2} \right\}\,.
		\end{align*}
		We denote
		\begin{align*}
		&\triangle \defn (\gamma+\xi-2+z)^2+4(z-1)(1-\gamma)(1-\xi)\,,\\
		&\triangle' \defn \frac{\mathrm{d}\triangle}{\mathrm{d}z}=2(z+\gamma+\xi-2)+4(1-\gamma)(1-\xi)\,.
		\end{align*}
		Then, using the expression~\eqref{EqnStietljesC} of $m_C$ and taking the derivative, it follows that
		\begin{align}
		\frac{\mathrm{d}m_C(z)}{\mathrm{d}z}-\frac{1-\gamma}{z^2} &=\frac{1-\frac{1}{2\sqrt{\triangle}}(2(z+\gamma+\xi-2)+4(1-\gamma)(1-\xi))}{2z(1-z)}\\
		&\quad+\frac{(z+\gamma+\xi-2-\sqrt{\triangle})(2z-1)}{2z^2(z-1)^2}+\frac{\gamma-1}{z^2}\\
		&=\frac{1}{2z^2(z-1)^2}[\triangle_1+(2\gamma\xi-\gamma-\xi)\triangle_2 -\triangle_3+\triangle_4],
		\end{align}
		where
		\begin{align*}
		\begin{cases}
		\triangle_1 = \frac{z^2(z-1)}{\sqrt\triangle}\\
		\triangle_2 = \frac{z(z-1)}{\sqrt\triangle}\\
		\triangle_3 = (2z-1)\sqrt{\triangle}\\
		\triangle_4 = z(1-z)+(z+\gamma+\xi-2)(2z-1)+2(\gamma-1)(z-1)^2.
		\end{cases}
		\end{align*}
		According to L'Hospital rule, 
		\begin{align}
		\label{EqnExpression2theta2}
		\gamma \theta_{2,h} = \lim_{z\rightarrow0}\frac{\triangle''_1+(2\gamma\xi-\gamma-\xi)\triangle''_2 -\triangle''_3+\triangle''_4}{2(12z^2-12z+2)}=\lim_{z\rightarrow0}\frac{\triangle''_1+(2\gamma\xi-\gamma-\xi)\triangle''_2 -\triangle''_3+\triangle''_4}{4}\,,
		\end{align}
		where $\triangle''_i$ denotes the second derivative of $\triangle_i$ with respect to~$z$. After some calculations, we find that
		\begin{align*}
		\triangle''_1|_{z=0}&=-\frac{2}{\xi-\gamma}\,,\\
		\triangle''_2|_{z=0}&=\frac{2}{\xi-\gamma}+\frac{4\gamma\xi-2\gamma-2\xi}{(\xi-\gamma)^3}\,,\\
		\triangle''_3|_{z=0}&=\frac{4(2\gamma\xi-\gamma-\xi)-1}{\xi-\gamma}+\frac{(2\gamma\xi-\gamma-\xi)^2}{(\xi-\gamma)^3}\,,\\
		\triangle''_4|_{z=0}&=2(2\gamma-1)\,.
		\end{align*}
		Using~\eqref{EqnExpression2theta2}, it follows that
		\begin{align*}
		\gamma \theta_{2,h} = \frac{1}{4} \left(\frac{-(2\gamma-1)^2}{\xi-\gamma}+\frac{(2\gamma\xi-\gamma-\xi)^2}{(\xi-\gamma)^3}\right) = \frac{\gamma(1-\gamma)(\gamma^2+\xi-2\gamma\xi)}{(\xi-\gamma)^3}\,,
		\end{align*}
		and finally, we obtain the claimed expression, that is, $\theta_{2,h} = \frac{(1-\gamma)(\gamma^2+\xi-2\gamma\xi)}{(\xi-\gamma)^3}$.
	\end{enumerate}
	
	\subsection{Proof of Theorem \ref{TheoremIHSHaar}}
	\label{ProofTheoremIHSHaar}
	\begin{proof}
		Let $\{S_t\}$ be a sequence of independent $m \times n$ Haar matrices, and let $\{x_t\}$ be the sequence of iterates generated by the update~\eqref{EqnIHSUpdate} with $\mu_t = \theta_{1,h}/\theta_{2,h}$ and $\beta_t=0$. Recall that we denote $\Delta_t = U^\top A (x_t - x^*)$, where $A = U \Sigma V^\top$ is a thin singular value decomposition of $A$. For $t \gre 0$, we have that
		\begin{align*}
		A \left(A^\top S^\top S A\right)^{-1} A^\top &= U \Sigma V^\top \left(V \Sigma U^\top S^\top S U \Sigma V^\top \right)^{-1} V \Sigma U^\top\\
		&= U \Sigma V^\top V \Sigma^{-1} (U^\top S^\top S U)^{-1} \Sigma^{-1} V V^\top \Sigma U^\top\\
		&= U (U^\top S^\top S U)^{-1} U^\top
		\end{align*}
		Multiplying both sides of the update formula~\eqref{EqnIHSUpdate} by $A$, subtracting $A x^*$ and using the normal equation $A^\top A x^* = A^\top b$, we find that
		\begin{align}
		\label{EqnRecError1}
		A (x_{t+1} - x^*) = \left(I_n - \mu_t U (U^\top S_t^\top S_t U)^{-1} U^\top \right) A (x_t - x^*)\,.
		\end{align}
		Multiplying both sides of~\eqref{EqnRecError1} by $U^\top$, using the definition of $\Delta_t$ and the fact that $U^\top U = I_d$, it follows that
		\begin{align*}
		\Delta_{t+1} &= U^\top \left(I_n - \mu_t U (U^\top S_t^\top S_t U)^{-1} U^\top \right) A(x_t-x^*)\\
		&= \left(U^\top - \mu_t U^\top U (U^\top S_t^\top S_t U)^{-1} U^\top \right) (Ax_t-x^*)\\
		&= \left(I_d - \mu_t (U^\top S_t^\top S_t U)^{-1}\right) \Delta_t\,,
		\end{align*}
		and then, taking the squared norm,
		\begin{align*}
		\|\Delta_{t+1}\|^2 = \Delta_t^\top \left(I_d - \mu_t (U^\top S_t^\top S_t U)^{-1}\right)^2 \Delta_t\,.
		\end{align*}
		Taking the expectation with respect to $S_t$ and using the independence of $S_t$ with respect to $S_0, \dots, S_{t-1}$, we obtain that
		\begin{align}
		\Exs_{S_t} \left[\|\Delta_{t+1}\|^2\right] &= \Delta_t^\top \Exs\left[\left(I_d - \mu_t (U^\top S_t^\top S_t U)^{-1}\right)^2\right] \Delta_t\\
		&= \Delta_t^\top \Big(I_d - 2 \mu_t\,\Exs\left[(U^\top S_t^\top S_t U)^{-1}\right] + \mu_t^2\,\Exs\left[(U^\top S_t^\top S_t U)^{-2}\right]\Big) \Delta_t \label{EqnRecError2}\,.
		\end{align}
		We write the spectral decomposition $U^\top S_t^\top S_t U = V \Sigma V^\top$ where $\Sigma$ is diagonal with positive entries $\lambda_{1}, \dots, \lambda_{d}$ and $V_t = [v_{1},\dots,v_{d}]$ is a $d \times d$ orthogonal matrix. The matrix $S_t U$ is distributed as the $m \times d$ upper-left block of an $n \times n$ Haar matrix. Therefore, $S_t U$ is right rotationally invariant, and so is the matrix $V$. It follows that $\lambda_i v_{ik} v_{i\ell} \overset{\mathrm{d}}{=} - \lambda_i v_{ik}v_{i\ell}$ for any index $i$ and any indices $k \!\neq\! \ell$. Then, for any $p \in \{1,2\}$ and any $k \neq \ell$, we have
		\begin{align*}
		\Exs\left[\left((U^\top S^\top S U)^{-p} \right)_{k\ell}\right] = \sum_{i=1}^d \Exs\left[\lambda_i^{-p} v_{ik} v_{i\ell} \right] = - \sum_{i=1}^d \Exs\left[\lambda_i^{-p} v_{ik} v_{i\ell}\right]\,,
		\end{align*}
		which implies that the off-diagonal term $\Exs\left[\left((U^\top S^\top S U)^{-p} \right)_{k\ell}\right]$ is equal to $0$. Further, by permutation invariance of the matrix $V$, we get that for any $k$,
		\begin{align*}
		\Exs\left[\left((U^\top S^\top S U)^{-p} \right)_{kk}\right]= \frac{1}{d} \trace \Exs\left[(U^\top S^\top S U)^{-p}\right]\,,
		\end{align*}
		or equivalently, $\Exs\left[(U^\top S^\top S U)^{-p}\right] = \theta_{p,n} I_d$ where $\theta_{p,n} \!\defn\! d^{-1} \trace \Exs\left[(U^\top S^\top S U)^{-p}\right]$. Then, using~\eqref{EqnRecError2}, it follows that
		\begin{align*}
		\Exs_{S_t} \left[\|\Delta_{t+1}\|^2\right] &= \Delta_t^\top \Big(I_d - 2 \mu_t\,\theta_{1,n} I_d + \mu_t^2\,\theta_{2,n} I_d \Big) \Delta_t\\
		&= (1-2 \mu_t \theta_{1,n} + \mu_t^2 \theta_{2,n}) \cdot \|\Delta_t\|^2\\
		&= \left(1-\frac{{\theta_{1,n}}^2}{\theta_{2,n}} + \left(\frac{\theta_{1,n}}{\sqrt{\theta_{2,n}}} - \mu_t \sqrt{\theta_{2,n}}\right)^2\right) \cdot \|\Delta_t\|^2\,.
		\end{align*}
		By induction, we further obtain
		\begin{align*}
		\frac{\Exs \left[\|\Delta_t\|^2\right]}{\|\Delta_0\|^2} = \prod_{j=0}^{t-1} \left(1-\frac{{\theta_{1,n}}^2}{\theta_{2,n}} + \left(\frac{\theta_{1,n}}{\sqrt{\theta_{2,n}}} - \mu_j \sqrt{\theta_{2,n}}\right)^2\right)\,.
		\end{align*}
		Taking the limit $n \to \infty$ and using the definition $\theta_{h,p} = \lim_{n\rightarrow\infty} \theta_{p,n}$ for $p\in\{1,2\}$, we find that
		\begin{align*}
		\lim_{n\to\infty}\,\frac{\Exs \left[\|\Delta_t\|^2\right]}{\|\Delta_0\|^2} = \prod_{j=0}^{t-1} \left(1-\frac{{\theta_{1,h}}^2}{\theta_{2,h}} + \left(\frac{\theta_{1,h}}{\sqrt{\theta_{2,h}}} - \mu_j \sqrt{\theta_{2,h}}\right)^2\right)\,.
		\end{align*}
		The above right-hand side is minimized at $\mu_j = \theta_{1,h} / \theta_{2,h}$ for all times steps $j \gre 0$, which yields the error formula
		\begin{align*}
		\lim_{n\to\infty}\,\frac{\Exs \left[\|\Delta_t\|^2\right]}{\|\Delta_0\|^2} = \left(1 - \frac{{{\theta_{1,h}}^2}}{\theta_{2,h}}\right)^t\,.
		\end{align*}
		Plugging-in the expressions of $\theta_{1,h}$ and $\theta_{2,h}$, we obtain the claimed convergence rate $\rho_h$.
		
		It remains to prove that $\rho_h$ is the best rate one may achieve with the update~\eqref{EqnIHSUpdate} along with Haar embeddings. It is actually an immediate consequence of Theorem~2 in~\cite{lacotte2019faster} whose assumptions (precisely, Assumption 1 in~\cite{lacotte2019faster}) are trivially satisfied by Haar embeddings.
		
	\end{proof}
	
	\subsection{Calculations of $\theta_{1,h}$ and $\theta_{2,h}$ for SRHT}
	\label{sec: moment SRHT}
	Our analysis proceeds in a way similar to the analysis of the Haar case, and we describe in this paragraph the main steps. Denote by $F_S$ the l.s.d.~of $U^\top S^\top SU$ and by $F_{S,n}$ its e.s.d.~As we did for the Haar case with the matrix $C_n$, we introduce here an auxiliary matrix $G_n$ whose e.s.d.~is related to $F_{S,n}$. Then, we characterize the $\eta$-transform $\eta_G$ of its l.s.d.~$F_G$. Our analysis for $\eta_G$ uses recent results on \emph{asymptotically liberating sequences} from free probability \cite{anderson2014asymptotically}. This technique has also been used in the prior work \cite{dobriban2019asymptotics}. Finally, we show that $\eta_G$ is equal to the $\eta$-transform $\eta_C$ of $F_C$, and we conclude that $F_S=F_h$. 
	
	Let $S \!=\! B H_n D P$ be the $n \times n$ SRHT matrix (before discarding the rows) as defined in Section \ref{SectionSRHT} in the paper, and $U$ be an $n \times d$ deterministic matrix with orthonormal columns. Note that whether we consider the zero rows or not in the matrix $S$, the matrix $U^\top S^\top S U$ remains the same, and so does its l.s.d.~The matrices $B, H_n$ and $D$ are all symmetric matrices, and they respectively satisfy $B^2 = B$, $H_n^2 = I_n$ and $D^2 = I_n$, and $P$ is also an orthogonal matrix. Then, we have that $S^\top S = P^\top DH_nBH_nDP$, and further,
	\begin{align*}
	(S^\top S)^2 &= P^\top DH_nBH_nDPP^\top DH_nBH_nDP = P^\top DH_nBH_nDP = S^\top S\,.
	\end{align*}
	We first have the following observation, whose proof is deferred to Appendix~\ref{AppendixLemmaHadamardEqualDist}.
	\begin{lemma} 
		\label{LemmaHadamardEqualDist}
		For $P$, $B$, $D$, $H_n$ and $U$ defined as above, we have the following equality in distribution
		\begin{align}
		U^\top (P^\top DH_n) B (HDP) U \stackrel{\mathrm{d}}{=} U^\top  (P^\top DH_nDP) B (P^\top DH_nDP) U\,.
		\end{align}
	\end{lemma}
	We now proceed with asymptotic statements, and we introduce the subscript $n$ to all matrices. We set $W_n \defn P_n^\top D_n H_n D_n P_n$. It holds that the matrix $U_n^\top W_n B_n W_n U_n$ has the same nonzero eigenvalues as $G_n \defn B_n W_n U_n U_n^\top W_n B_n$, so that we first find the l.s.d.~of the matrix $G_n$. The reader may notice that $G_n$ plays a similar role in the analysis of the SRHT case, to that of the matrix $C_n$ in the analysis of the Haar case.
	
	The following result states the asymptotic freeness of the matrices $B_n$ and $W_n U_n U_n^\top W_n$. Its proof follows directly from Corollaries~3.5 and~3.7 by~\cite{anderson2014asymptotically}. 
	\begin{lemma}
		\label{LemmaFreenessHadamardMatrix}
		Let $B_n, W_n, U_n$ be defined as above. Then, the matrices $\{B_n, W_n U_n U_n^\top W_n\}$ are asymptotically free in the limit of the non-commutative probability spaces of random matrices. Consequently, the e.s.d.~of the matrix $G_n=B_n W_n U_n U_n^\top W_n B_n$ converges to the freely multiplicative convolution of the l.s.d.~$F_B$ of $B_n$ and the l.s.d.~$F_U$ of $U_n U_n^\top$, that is, $G_n$ has l.s.d.~given by $F_G = F_B \boxtimes F_U$.
	\end{lemma} 
	Since the density of the l.s.d.~$F_B$ is $f_B = \xi\delta_1+(1-\xi)\delta_0$ and and the density of $F_U$ is $f_U = \gamma \delta_1 + (1-\gamma) \delta_0$, we have that the $S$-transforms $S_B$ of $F_B$ and $S_U$ of $F_U$ are respectively equal to $S_{B}(y)=\frac{y+1}{y+\xi}$ and $S_U(y) = \frac{y+1}{y+\gamma}$. From Lemma~\ref{LemmaFreenessHadamardMatrix}, it follows that the $S$-transform $S_G$ of $F_G$ is the product of $S_B$ and $S_U$, i.e.,
	\begin{align}
	\label{EqnFreenessmuCmuXmuB}
	S_{G}(y) = S_{U}(y)S_{B}(y) = \frac{(y+1)^2}{(y+\xi)(y+\gamma)}\,.
	\end{align}
	First, note that using their respective definitions, the $S$-transform of $F_G$ and its $\eta$-transform $\eta_G$ are related by the equation $\eta_{G}\!\left(-\frac{y}{y+1}S_{G}(y)\right) = y+1$. Plugging-in the expression~\eqref{EqnFreenessmuCmuXmuB} of $S_G(y)$ into the latter equation, we obtain that 
	\begin{align*}
	\eta_{G}\!\left(-\frac{y(y+1)}{(y+\gamma)(y+\xi)}\right) = y+1\,.
	\end{align*}
	Letting $z=-\frac{(y+\gamma)(y+\xi)}{y(y+1)}$ and using the relationship~\eqref{EqnEtaTransform} between the Stieltjes and $\eta$-transforms, we find that the Stieltjes transform $m_G$ of $G$ is equal to
	\begin{align*}
	m_G(z) = \frac{z+\gamma+\xi-2 - \sqrt{g(z)}}{2z(1-z)}\,,
	\end{align*}
	where $g(z)=(\gamma+\xi-2+z)^2+4(z-1)(1-\gamma)(1-\xi)$. Hence, we get that $m_G(z) = m_C(z)$, that is, $F_G=F_C$.
	
	Further, the matrix $G_n$ has the same non-zero eigenvalues as the matrix $U_n^\top W_n B_n W_n U_n$ which, according to Lemma~\ref{LemmaHadamardEqualDist}, is equal in distribution to $U_n^\top S_n^\top S_n U_n$. Denote by $\lambda_1, \dots, \lambda_{\widetilde d}$ the non-zero eigenvalues of $U_n^\top S_n^\top S_n U_n$, where $\widetilde d$ is itself a random number due to the randomness of non-zero rows $\widetilde m$. Hence, the e.s.d~$F_{G,n}$ of $G_n$ and the e.s.d.~$F_{S,n}$ of $U_n^\top S_n^\top S_n U_n$ satisfy (see Appendix \ref{sec:appidentity})
	%
	%
	\begin{align}
	F_{G_n}(x) \overset{\mathrm{d}}{=} \left(1-\frac{d}{n}\right) \mathbf{1}_{\{x \gre 0\}} + \frac{d}{n} F_{S,n}(x) \label{eq:esdequality}\,.
	\end{align}
	Thus, we obtain that $F_{S,n}$ converges weakly almost surely to the distribution 
	\begin{align}
	F_S(x) \defn \frac{1}{\gamma}\left(F_G(x)-(1-\gamma)\mathbf{1}_{\{x \gre 0\}}\right) = \frac{1}{\gamma}\left(F_C(x)-(1-\gamma)\mathbf{1}_{\{x \gre 0\}}\right)\,.
	\end{align}
	The latter expression is equal to $F_h(x)$ according to~\eqref{EqnFhFc}, so that $F_S(x) = F_h(x)$. The analysis of the traces of the expected first and second inverse moments only involves the limiting distribution (we refer the reader to the proof of the expressions of $\theta_{1,h}$ and $\theta_{2,h}$, in Section~\ref{ProofTraceCalculationsHaar}). Due to the equality $F_h=F_S$, they remain the same with SRHT matrices, which concludes the proof of Lemma~\ref{TheoremSRHTlsd}.

	\subsection{Proof of Theorem~\ref{TheoremIHSSRHT} and \ref{TheoremIHSSRHTb}}
	\label{ProofTheoremIHSSRHT}
	
	Let $\{S_t\}$ be a sequence of independent $m \times n$ SRHT matrices, and let $\{x_t\}$ be the sequence of iterates generated by the update~\eqref{EqnIHSUpdate} with $\mu_t = \theta_{1,h}/\theta_{2,h}$ and $\beta_t=0$. Denote $\Delta_t = U^\top A(x_t - x^*)$ the sequence of error vectors. The proof follows exactly the same lines as for Theorem~\ref{TheoremIHSSRHT} up to the relationship~\eqref{EqnRecError2}, which we recall here,
	\begin{align}
	\Exs_{S_t} \left[\|\Delta_{t+1}\|^2\right] = \Exs_{S_t}\left[\Delta_t^\top \left(I_d - \mu_t\,(U^\top S_t^\top S_t U)^{-1} \right)^2 \Delta_t\right]\,.
	\end{align}
	Denote $Q_t = I_d - \mu_t\,(U^\top S_t^\top S_t U)^{-1}$. It holds that $\Delta_{t+1}=Q_t \Delta_t$ as previously shown. Hence, by induction, we obtain that
	\begin{align}
	\Exs \left[\|\Delta_t\|^2\right] &= \trace \Exs \left[Q_0\dots Q_{t-1} Q_{t-1} \dots Q_0 \Delta_0 \Delta_0^\top \right]\,. \label{eq:SRHTtraceFormula}
	\end{align}
	Using the independence of $\Delta_0$ and the $Q_i$, and the assumption $\Exs\left[\Delta_0 \Delta_0^\top \right]=I_d / d$, it follows that
	\begin{align}
	\Exs \left[\|\Delta_t\|^2\right] &= \frac{1}{d} \trace\Exs \left[ Q_1\dots Q_{t-1}Q_{t-1} \dots Q_0^2 \right]\,.\label{eq:SRHTtraceFormulaNoDelta} 
	\end{align}
	It holds that the matrix $Q_0^2$ is asymptotically free from $Q_{t-1} \dots Q_1$. Therefore, using the trace decoupling relation~\eqref{EqnTraceDecoupling}, we have that
	\begin{align*}
	\lim_{n \to \infty} \Exs \left[\|\Delta_t\|^2\right] &= \lim_{n \to \infty} \frac{1}{d} \trace \Exs \left[Q_1\dots Q_{t-1} Q_{t-1} \dots Q_0^2 \right]\\
	&= \lim_{n \to \infty} \frac{1}{d}\trace \Exs \left[ Q_0^2 \right] \cdot \lim_{n \to \infty} \frac{1}{d} \trace \Exs \left[ Q_2 \dots Q_{t-1} Q_{t-1} \cdots Q_1^2 \right]\,.
	\end{align*}
	Note that $\lim_{n \to \infty} \frac{1}{d}\trace \Exs \left[ Q_0^2 \right] = (1-2\mu_0 \theta_{1,h} + \mu_0^2 \theta_{2,h})$. Repeating the same asymptotic freeness argument between $Q_1^2$ and $Q_{t-1} \dots Q_2$ and plugging-in $\mu_j = \theta_{1,h}/\theta_{2,h}$, we finally obtain the claimed result,
	\begin{align*}
	\lim_{n \to \infty} \Exs\left[\|\Delta_{t+1}\|^2\right] &= \prod_{j=0}^{t-1} \left(1-\mu_j \theta_{1,h} + \mu_j^2 \theta_{2,h} \right)\\
	& = \left(1-\frac{\theta_{1,h}^2}{\theta_{2,h}}\right)^t\,.
	\end{align*}

	The proof of Theorem \ref{TheoremIHSSRHTb} immediately follows from an alternative upper-bound on the expression \eqref{eq:SRHTtraceFormula} for the norm of the error. In particular, we note that
	\begin{align*}
	\Exs \left[\|\Delta_t\|^2\right] &= \trace \Exs \left[Q_0\dots Q_{t-1} Q_{t-1} \dots Q_0 \Delta_0 \Delta_0^\top \right]\\
	&\le \|\Delta_0\Delta_0^\top\|_2\trace \Exs \left[Q_0\dots Q_{t-1} Q_{t-1} \dots Q_0 \right]\\
	&= d\|\Delta_0\|^2_2 \frac{1}{d}\trace \Exs \left[Q_0\dots Q_{t-1} Q_{t-1} \dots Q_0 \right].
	\end{align*}
	We then combine the earlier expression \eqref{eq:SRHTtraceFormulaNoDelta} with the above upper-bound and complete the proof.

	\begin{remark}
		\label{RemarkOptimalitySRHT}
		In view of equations (4-6) in \cite{anderson2014asymptotically}, one can show that asymptotic freeness between $U^\top S^\top S U$ and a rank-one matrix $v v^\top$ holds provided that $\|v\|_2 < \infty$ as the dimensions grow to infinity. One could then wonder whether such a result can be applied to our setting, in order to remove the assumption $\Exs \Delta_0 \Delta_0^\top = \frac{1}{d} \cdot I_d$. Using~\eqref{eq:SRHTtraceFormula}, dividing by $ \Exs\|\Delta_0\|^2$ and denoting $\wtilde \Delta_0 = \frac{\Delta_0}{\sqrt{\Exs \|\Delta_0\|^2 / d}}$, we get
		\begin{align*}
		\frac{\Exs \|\Delta_t\|^2}{\Exs \|\Delta_0\|^2} = \frac{1}{d} \, \trace \Exs \left[Q_0\dots Q_{t-1} Q_{t-1} \dots Q_0 \wtilde \Delta_0 \wtilde \Delta_0^\top \right]\,.
		\end{align*}
		Provided we have asymptotic freeness between $\wtilde \Delta_0 \wtilde \Delta_0^\top$ and $Q_0 \dots Q_{t-1} Q_{t-1} \dots Q_0$, then we have
		\begin{align*}
		\lim_{n \to \infty} \, \frac{\Exs \|\Delta_t\|^2}{\Exs \|\Delta_0\|^2} = \lim_{n \infty}\frac{1}{d} \, \trace \Exs \left[Q_0\dots Q_{t-1} Q_{t-1} \dots Q_0\right] \cdot \lim_{n \infty}\frac{1}{d} \trace \Exs \left[\wtilde \Delta_0 \wtilde \Delta_0^\top \right]
		\end{align*}
		According to our previous analysis, the term $\lim_{n \infty}\frac{1}{d} \, \trace \Exs \left[Q_0\dots Q_{t-1} Q_{t-1} \dots Q_0\right]$ is equal to $(1-\frac{\theta_{1,h}^2}{\theta_{2,h}})^t$. On the other hand, the term $\lim_{n \infty}\frac{1}{d} \trace \Exs \left[\wtilde \Delta_0 \wtilde \Delta_0^\top \right]$ is equal to $1$, so that we would get the claimed result. But, for asymptotic freeness to hold between $\wtilde \Delta_0 \wtilde \Delta_0^\top$ and $Q_0 \dots Q_{t-1} Q_{t-1} \dots Q_0$, we need $\|\wtilde \Delta_0\| < \infty$, and this assumption seems too strong: for instance, if $\Delta_0$ is deterministic, then $\|\wtilde \Delta_0\| = \sqrt{d}$ which is unbounded as the dimensions grow to infinity.
	\end{remark}

	\section{Proofs of the auxiliary results}

	\subsection{Proof of the bounds on the support of $F_h$ (Lemma \ref{lem: support of Fh})}
	\label{ProofLemmaSupportFh}
	\begin{proof}
		We show that the support of $F_h$ satisfies
		\begin{align*}
		\inf \, \mathrm{supp}(F_h) \gre \frac{\left(1-\sqrt{\rho_g}\right)^2}{(1+\frac{1}{\sqrt{\xi}})^2}\,.
		\end{align*}
		Let $S$ be an $m \times n$ Haar matrix, $U$ an $n \times d$ deterministic matrix with orthonormal columns, and $S_g$ be an $m \times n$ matrix independent of $S$, with i.i.d.~entries $\mathcal{N}(0,1/m)$. Write $S_g=\Omega_\ell \Sigma \Omega_r$ a singular value decomposition of $S_g$. It holds that $\Omega_\ell$ is an $m \times m$ Haar matrix, independent of the $m \times m$ diagonal matrix of singular values $\Sigma$, and $\Omega_r \overset{\mathrm{d}}{=} S$, so that $\Omega_\ell \Sigma S \overset{\mathrm{d}}{=} S_g$. Further, the operator norm of $\Sigma$ satisfies $\lim_{n \to \infty} \|\Sigma\|_2 = \left(1+\frac{1}{\sqrt{\xi}}\right)$ almost surely. Then,
		\begin{align*}
		\sigma_{\min}(SU) = \min_{\|x\|=1} \|SUx\| &\gre \min_{\|x\|=1} \frac{\|\Sigma S U x\|}{\|\Sigma\|_2}\\
		& = \frac{1}{{\|\Sigma\|_2}} \cdot \min_{\|x\|=1}  \|\Omega_\ell  \Sigma S U x\|\,.
		\end{align*}
		Almost surely, $\min_{\|x\|=1} \|\Omega_\ell \Sigma S x\| \to (1-\sqrt{\rho_g})$ as $n \to \infty$. Thus, almost surely, $\liminf_{n \to \infty}\sigma_{\min}(SU) \gre \frac{\left(1-\sqrt{\rho_g}\right)}{(1+\frac{1}{\sqrt{\xi}})}$, which yields the claimed lower bound on the support of $F_h$.
	\end{proof}

	\subsection{Characterization of the e.s.d. of $C_n$}
	\label{sec: esd of Cn}
	
	Recall the definition~\eqref{EqnCn} of the matrix $C_n$,
	\begin{align*}
	C_n = \left(\begin{array}{cc}I_m & 0\\0 & 0\end{array}\right)
	\mathbb{W}_n
	\left(\begin{array}{cc} I_d & 0\\0&0\end{array}\right)
	\mathbb{W}_n^\top
	\left(\begin{array}{cc}I_m & 0\\0 & 0\end{array}\right)\,.
	\end{align*}

	We leverage Theorem 4.11 from~\cite{couillet2011random}, which we recall for the sake of completeness. 
	\begin{theorem}[Theorem 4.11,~\cite{couillet2011random}]
		\label{TheoremCouillet}
		Let $D_n \in \real^{n \times n}$ and $T_n \in \real^{n \times n}$ be diagonal non-negative matrices, and $\mathbb{W}_n \in \real^{n \times n}$ be a Haar matrix. Denote $F_D$ and $F_T$ the respective l.s.d.~of $D_n$ and $T_n$. Denote $C_n$ the matrix $C_n \defn D_n^\frac{1}{2} \mathbb{W}_n T_n \mathbb{W}_n^\top D_n^\frac{1}{2}$. Then, as $n$ tends to infinity, the e.s.d.~of $C_n$ converges to $F$ whose $\eta$-transform $\eta_F$ satisfies
		\begin{align*}
		\eta_F(z) &= \int \frac{1}{z \gamma(z) x + 1} \,\mathrm{d}F_D(x)\,,\\
		\gamma(z) &= \int \frac{x}{\eta_F(z) + z \delta(z) x} \, \mathrm{d}F_T(x)\,,\\
		\delta(z) &= \int \frac{x}{z \gamma(z) x + 1} \mathrm{d}F_D(x)\,.
		\end{align*}
	\end{theorem} 
	The e.s.d.~of $\left(\begin{array}{cc} I_d & 0\\0 & 0\end{array}\right)$ converges to the distribution $F_\gamma$ with density $\gamma \delta_1 + (1-\gamma) \delta_0$, and the e.s.d.~of $\left(\begin{array}{cc}I_m & 0\\0 & 0\end{array}\right)$ converges to the distribution $F_\xi$ with density $\xi\delta_1+(1-\xi)\delta_0$. Then, according to Theorem~\ref{TheoremCouillet}, the e.s.d. of $C_n$ converges to a distribution $F_C$, whose $\eta$-transform $\eta_C$ is solution of the following system of equations,
	\begin{align}
	\eta_C(z) &= \int\frac{1}{z\gamma(z)x+1}\,\mathrm{d} F_\xi(x)\,,\\
	\gamma(z) &= \int\frac{x}{\eta_C(z)+z\delta(z)x}\,\mathrm{d} F_\gamma(x)\,,\\
	\delta(z) &= \int\frac{x}{z\gamma(z)x+1}\,\mathrm{d} F_\xi(x)\,.
	\end{align}
	Plugging the above expressions of $F_\xi$ and $F_\gamma$ into the above equations, and after simplification, we obtain that $\eta_C$ is solution of the following second-order equation
	\begin{align}
	\label{EqnetaCfixedpoint}
	\eta_C(z) = (1-\gamma) + \frac{\gamma}{1 + z\left(1+\frac{\xi-1}{\eta_{C}(z)}\right)}\,,
	\end{align}
	Plugging the relationship~\eqref{EqnEtaTransform} between the Stieltjes and $\eta$-transforms into~\eqref{EqnetaCfixedpoint}, we find that
	\begin{align}
	\label{Eqnmcundetermined}
	m_C(z) = \frac{z+\gamma+\xi-2- \sqrt{g(z)}}{2z(1-z)}\,,
	\end{align}
	where $g(z)=(\gamma+\xi-2+z)^2+4(z-1)(1-\gamma)(1-\xi)$, and we choose the branch of the square-root such that $m_C(z) \in \mathbb{C}^+$ for $z \in \mathbb{C}^+$, $m_C(z) \in \mathbb{C}^-$ for $z \in \mathbb{C}^-$ and $m_C(z) > 0$ for $z < 0$.

	\subsection{Proof of Lemma~\ref{LemmaHadamardEqualDist}}
	\label{AppendixLemmaHadamardEqualDist}
	\begin{proof}
		Note that both $B$ and $D$ are diagonal matrices whose diagonal entries are i.i.d.~random variables, and $P$ is a permutation matrix. Define $\tilde{B}=P B P^\top$ and $\tilde{D}=P^\top DP$, then
		we have
		$$\tilde{B}\stackrel{d}{=}B,\quad\tilde{D}\stackrel{d}{=}D$$
		and
		\begin{align}
		\label{EqnDP=PD}
		DP=P\tilde{D},\quad P^\top D=\tilde{D}P^\top\,.
		\end{align}
		It follows that
		\begin{align*}
		U^\top P^\top DH_n DPBP^\top DH_n DP U&=U^\top P^\top DH_n P\tilde{D}B\tilde{D}P^\top H_n DP U\\
		&=U^\top  P^\top DH_n PB\tilde{D}^2 P^\top H_n DP U\\
		&=U^\top  P^\top DH_n P BP^\top H_n DP U\\
		&=U^\top  P^\top DH_n\tilde{B} H_n DP U\\
		&\stackrel{d}{=}U^\top P^\top DH_n B H_n DP U,
		\end{align*}
		where the first equation follows from \eqref{EqnDP=PD}, the second equation holds because $\tilde{D}$ and $B$ are diagonal so they commute, while the third equation holds because $\tilde{D}^2 = I_n$.
	\end{proof}

	\subsection{Proof of the identity~\eqref{eq:esdequality}}
	\label{sec:appidentity}
	We note that
	\begin{align*}
	F_{G_n}(x) &\overset{\mathrm{d}}{=} \left(1-\frac{\widetilde d}{n}\right) \mathbf{1}_{\{x \gre 0\}} + \frac{1}{n} \sum_{j=1}^{\widetilde d} \mathbf{1}_{\{x \gre \lambda_j\}}\\
	&= \left(1-\frac{\widetilde d}{n}\right) \mathbf{1}_{\{x \gre 0\}} + \frac{d}{n} \cdot \frac{1}{d} \sum_{j=1}^{\widetilde d} \mathbf{1}_{\{x \gre \lambda_j\}}\\
	&= \left(1-\frac{\widetilde d}{n}\right) \mathbf{1}_{\{x \gre 0\}} + \frac{d}{n} \left(F_{S,n}(x) - \left(\frac{d-\widetilde d}{d}\right) \mathbf{1}_{\{x \gre 0\}}\right)\\
	&= \left(1-\frac{d}{n}\right) \mathbf{1}_{\{x \gre 0\}} + \frac{d}{n} F_{S,n}(x)\,,
	\end{align*}
	which proves \eqref{eq:esdequality}.

\end{appendix}

\end{document}